\documentclass[12pt, a4paper]{amsart}

\usepackage[left=20mm, right=20mm, top=30mm, bottom=30mm, marginparwidth=80pt]{geometry}

\usepackage{bm}
\usepackage[authoryear,longnamesfirst,sort,numbers]{natbib}
\usepackage{amsfonts}
\usepackage{amsthm, amsmath, amssymb}
\usepackage{graphicx}
\usepackage{xcolor}
\usepackage{booktabs} 
\usepackage{subcaption} 
\usepackage{float}
\usepackage{lipsum}
\usepackage{tikz}
\numberwithin{equation}{section}

\usepackage{enumitem}

\usepackage{mathtools}
\mathtoolsset{showonlyrefs}

\usepackage{hyperref}
\usepackage{bbm}

\usepackage{xifthen}

\newcommand{\RR}{\mathbb{R}}
\newcommand{\NN}{\mathbb{N}}
\newcommand{\diag}{\mathrm{diag}}

\renewcommand{\Re}{\mathrm{Re}\,}

\newcommand{\CSS}{\bar{X}}
\newcommand{\Cun}{\bar{\un}}
\newcommand{\Cus}{\bar{\us}}
\newcommand{\Cuf}{\bar{\uf}}

\newcommand{\NSS}{\tilde{X}}

\newcommand{\Nun}{\tilde{\un}}
\newcommand{\Nus}{\tilde{\us}}
\newcommand{\Nuf}{\tilde{\uf}}

\newcommand{\un}{u}
\newcommand{\us}{v}
\newcommand{\uf}{w}

\newcommand{\Ds}{D_v}
\newcommand{\Df}{D_w}

\newcommand{\mn}{m_{u}}  
\newcommand{\ms}{m_{v}}  
\newcommand{\mfa}{m_{w}} 
\newcommand{\md}{m_{d}}  

\newcommand{\fn}{f}
\newcommand{\fs}{g}
\newcommand{\ffa}{h}

\def\mf{\mu_1}

\def\ml{\mu_2}
\def\me{\mu_3}

\renewcommand{\L}{{\mathcal{L}}}

\newcommand{\W}[2]
{{
		W^{#1,#2}_
		{
			\ifthenelse
			{
				\equal{#1}{2}
			}
			{
				\nu
			}
			{
			}
		}(\Omega)
}}

\newcommand{\Lp}[1]
{{
		L^
		{
			\ifthenelse
			{
				\isempty{#1}
			}
			{
				p
			}
			{
				#1
			}
		}(\Omega)
}}

\definecolor{lime}{HTML}{A6CE39}
\DeclareRobustCommand{\orcidicon}{%
	\begin{tikzpicture}
		\draw[lime, fill=lime] (0,0) 
		circle [radius=0.16] 
		node[white] {{\fontfamily{qag}\selectfont \tiny ID}};
		\draw[white, fill=white] (-0.0625,0.095) 
		circle [radius=0.007];
	\end{tikzpicture}
	\hspace{-2mm}
}

\newenvironment{nalign}
{\begin{equation}\begin{aligned}}
		{\end{aligned}\end{equation}\ignorespacesafterend}

\def\be{\begin{equation}}
\def\ee{\end{equation}}
\def\bpm{\begin{pmatrix}}
\def\epm{\end{pmatrix}}

\def\lp{\left(}
\def\rp{\right)}

\def\del{\partial}
\def\grad{\nabla}
\DeclareMathOperator{\tr}{tr}

\newtheorem{theorem}{\bf Theorem}[section]

\newtheorem{corollary}[theorem]{\bf Corollary}
\newtheorem{proposition}[theorem]{\bf Proposition}
\newtheorem{definition}[theorem]{\bf Definition}
\newtheorem{lemma}[theorem]{\bf Lemma}
\newtheorem{remark}[theorem]{\bf Remark}
\newtheorem{example}[theorem]{\bf Example}
\newtheorem{assumption}[theorem]{\bf Assumption}

\title[DDI and discontinuous patterns with multiple diffusion scales]{Instability mechanisms and far-from-equilibrium patterns in reaction-diffusion-ODE systems with multiple diffusion scales}

\author[F. Münnich]{Finn Münnich \href{https://orcid.org/0009-0007-9384-2002}{\orcidicon}}
\address[F. Münnich]{Institute for Mathematics, Heidelberg University, Im Neuenheimer Feld 205, 69120 Heidelberg, Germany\\
\href{https://orcid.org/0009-0007-9384-2002}{orcid.org/0009-0007-9384-2002}}
\email{finn.muennich@stud.uni-heidelberg.de}

\author[S. Cygan]{Szymon Cygan \href{https://orcid.org/0000-0002-8601-829X}{\orcidicon}}
\address[S. Cygan]{Institute for Mathematics, Heidelberg University, Im Neuenheimer Feld 205, 69120 Heidelberg, Germany\\
Instytut Matematyczny, Uniwersytet Wroc\l{}awski, pl. Grunwaldzki 2/4, \hbox{50-384} Wroc\l{}aw, Poland \\ 
\href{https://orcid.org/0000-0002-8601-829X}{orcid.org/0000-0002-8601-829X}}
\email{szymon.cygan@uni-heidelberg.de}
\urladdr {http://www.math.uni.wroc.pl/~scygan}

\author[T. André]{Théo André \href{https://orcid.org/0009-0001-7320-1941}{\orcidicon}}
\address[T. André]{Institute for Mathematics, Heidelberg University, Im Neuenheimer Feld 205, 69120 Heidelberg, Germany\\
\href{https://orcid.org/0009-0001-7320-1941}{orcid.org/0009-0001-7320-1941}}
\email{theo.andre@uni-heidelberg.de}

\author[A. Marciniak-Czochra]{Anna Marciniak-Czochra \href{https://orcid.org/0000-0002-5831-6505}{\orcidicon}}
\address[A. Marciniak-Czochra]{Institute for Mathematics and IWR, Heidelberg University, Im Neuenheimer Feld 205, 69120 Heidelberg, Germany\\ 
\href{https://orcid.org/0000-0002-5831-6505}{orcid.org/0000-0002-5831-6505}}
\email{anna.marciniak@iwr.uni-heidelberg.de}
\urladdr{\href{https://biostruct.iwr.uni-heidelberg.de/folder_people/Anna.Marciniak/index.html}{https://biostruct.iwr.uni-heidelberg.de/folder\_people/Anna.Marciniak/index.html}}

\date{}

\begin{document}


\begin{abstract}
   Systems coupling diffusive and nondiffusive components exhibit mathematical phenomena absent in classical reaction-diffusion equations, arising from the singular nature of the vanishing-diffusion limit. These include  diffusion-driven instability (DDI) without stable Turing patterns and far-from-equilibrium stationary solutions with jump discontinuities. However, existing theory for diffusive--nondiffusive systems has so far been restricted to the case of a single diffusive component. Here, we extend the analysis to systems with multiple diffusion scales and show that they give rise to new pattern-forming scenarios. In particular, instability may also originate from coupled nondiffusive and slowly diffusing subsystems, yielding simple sufficient conditions for DDI in systems with an arbitrary number of components. We fully classify all possible sources of DDI in the minimal setting of one nondiffusive and two diffusive components and show that some of these mechanisms disappear under quasi-steady-state reduction.
Beyond the onset of instability, we establish the existence of far-from-equilibrium stationary patterns with jump discontinuities in the nondiffusive variables. The theory is illustrated using a receptor-based model together with stability analysis and numerical simulations. Together, our results show how multiple diffusion scales enrich both the mechanisms of DDI and the resulting pattern-forming behaviour in diffusive-nondiffusive systems.

\end{abstract}

\maketitle

\section{Introduction}\label{sec:intro}

A central question in pattern formation is how spatially heterogeneous structures emerge from spatially homogeneous steady states. One of the principal mechanisms underlying this phenomenon is diffusion-driven instability (DDI), also known as Turing instability \cite{turing1952chemical}. Near the instability threshold, it gives rise to Turing patterns, typically small-amplitude spatially heterogeneous solutions near the homogeneous equilibrium. Coupling diffusive and nondiffusive components fundamentally changes this picture. Such models are commonly referred to as reaction-diffusion-ODE systems, where the ODE refers to spatially dependent evolution of the nondiffusive variables. Diffusive-nondiffusive systems arise naturally in mathematical models of spatially structured biological \cite{Klika12,Takagi21_2,He19}, chemical and ecological processes \cite{Iuorio2023,marasco2014vegetation}. In many applications, only a subset of the species diffuses, while the remaining components remain spatially localized. Typical examples include receptor-mediated signalling, morphogenesis and tumour growth, where diffusive cell-cell communication is coupled to intracellular processes \cite{Zhang2025,Marciniak06,GALLAY2022103387}. In some cases, reaction-diffusion-ODE systems can be rigorously derived from underlying microscopic descriptions using homogenization techniques \cite{marciniakptashnyk2008,marciniak2012stronghomog}. The singular nature of the vanishing-diffusion limit gives rise to mathematical phenomena absent in classical reaction-diffusion equations, such as emergence of patterns with jump discontinuities \cite{Cygan2023,Akagi2024,Haerting17,Guo2025} or concentration phenomena \cite{HMC14,marciniak2018dynamical,marciniak2016blowup}, which are characteristic of far-from-equilibrium regimes. These mathematical features have motivated the development of a theory of pattern formation that differs fundamentally from the classical reaction-diffusion framework.

So far, this theory has been developed almost exclusively for systems with a single diffusive component. In this setting, autocatalysis in the nondiffusive subsystem provides a necessary and sufficient condition for diffusion-driven instability \cite{Marciniak15,Cygan2022}. However, this mechanism simultaneously destabilizes all Turing patterns bifurcating from the homogeneous equilibrium. Consequently, the resulting spatial patterns are necessarily far-from-equilibrium and may exhibit singular structures such as jump discontinuities. 

This limitation raises the question of how pattern formation changes when several components diffuse on different diffusion scales. Multiple diffusion scales arise naturally in a broad range of applications. One example is provided by ecological models of vegetation patterns, which combine nondiffusive autotoxic compounds with slowly diffusing vegetation biomass and rapidly diffusing soil water \cite{marasco2014vegetation,Veerman2024_1}. Receptor-based models couple membrane-bound receptors with diffusible signalling molecules acting on different diffusion scales \cite{Klika12,Zhang2025,Marciniak06,mou2006generation}. Related scale separation also appears in tumor invasion models through state-dependent diffusion \cite{GALLAY2022103387}. 

In this work, we develop a general theory of pattern formation in diffusive-nondiffusive systems with multiple diffusion scales. Our analysis shows that introducing additional diffusion scales fundamentally enlarges the class of admissible instability mechanisms while preserving the possibility of far-from-equilibrium pattern formation. The paper is organised as follows. Section \ref{sec:general_theory} establishes sufficient conditions for DDI in general diffusive--nondiffusive systems based on the decomposition into nondiffusive, slow- and fast-diffusive components. Section \ref{sec:special_case} studies the minimal three-component setting, deriving necessary conditions for DDI and identifying an instability mechanism generated solely by the diffusive subsystem. Section \ref{sec:existence_jump} extends the existence theory for discontinuous stationary solutions developed in \cite{Cygan2023} to the present multiscale setting. Finally, Section \ref{sec:Application} illustrates the theory using a receptor-based model.

\section{Mathematical setting}\label{sec:setting}
To capture this multiscale structure in a unified framework, we consider the following general class of diffusive-nondiffusive systems with \mbox{$m = \mn + \ms + \mfa$}
components of the form
\begin{align}\label{eq:full_system}
    \left\{
    \begin{aligned}
        \partial_t \un &= \fn(\un,\us,\uf),\\
        \partial_t \us &= \Ds \Delta \us + \fs(\un,\us,\uf),\\
        \partial_t \uf &= \Df \Delta \uf + \ffa(\un,\us,\uf)
    \end{aligned}
    \right.
\end{align}
posed in a bounded domain $\Omega\subset\RR^n$ with boundary $\partial\Omega\in C^2$.
Here, $\un$ denotes the $\mn$ nondiffusing components, $\us$ the $\ms$ slow-diffusing components, and $\uf$ the $\mfa$ fast-diffusing components. The diffusion matrices $\Ds$ and $\Df$ are diagonal with strictly positive entries, while the nondiffusing components have vanishing diffusion. Homogeneous Neumann boundary conditions are imposed on all diffusive species. The reaction terms $F=(\fn,\fs,\ffa)$ are assumed to be of class $C^2$, although weaker regularity is sufficient for the main results \cite{KMMü23}.

Let $\CSS=(\Cun,\Cus,\Cuf)$ be a constant steady state of \eqref{eq:full_system}. We say that $\CSS$ undergoes diffusion-driven instability (DDI) if it is stable for the corresponding kinetic system ($\Ds=\Df=0$) but becomes unstable after diffusion is introduced. For reaction--diffusion--ODE systems, nonlinear (in)stability of a bounded steady state follows directly from its linear (in)stability \cite[Theorems 3.3 and 3.7]{KMMü23}. Linearizing \eqref{eq:full_system} around a bounded steady state $\NSS$ yields the linearised operator
\[
\L=D\Delta+J,
\]
where $D=\diag(0,\Ds,\Df)$ and $J=\nabla_{(\un,\us,\uf)}F(\NSS)$ is the Jacobian of the reaction terms evaluated at $\NSS$. Consequently, the analysis of DDI for a constant steady state reduces to the spectral properties of $\L$, whose spectral bound is
\[
s(\L)=\sup_{\lambda\in\sigma(\L)}\Re\lambda.
\]

For a bounded steady state $\NSS$, the spectrum of $\mathcal{L}$ on the Banach space $L^p(\Omega)^{\mn+\ms+\mfa}$ or $L^\infty(\Omega)^{\mn}\times L^p(\Omega)^{\ms+\mfa}$ can be decomposed as
\begin{align}\label{eq:Spectrum_of_general_L}
    \sigma(\L) = \sigma\big(\nabla_{\un} \fn (\NSS)\big) \cup \sigma_p(\L),
\end{align}
under the assumption $\Ds, \Df >0$ \cite[Section 5.22]{Kowall_2}, \cite[Proposition 2.3]{KMMü23}.
In the case of a linearization at a constant steady state $\CSS$ and $p=2$, the point spectrum is expressed in terms of the eigenvalues $(\lambda_j)_{j\ge0}$ of the Laplace operator $-\Delta$ with Neumann boundary conditions:
\begin{align}\label{eq:eigenvalues_at_constant_steady_state}
    \sigma_p(\L) = \bigcup_{j=0}^\infty \sigma(-\lambda_j D + J) = \bigcup_{j=0}^\infty \big\{\lambda\in\mathbb{C} \mid \det(\lambda I + \lambda_j D - J) = 0\big\}.
\end{align}
For later reference, we define the following submatrices of $J = \nabla_{(\un, \us, \uf)}F(\CSS)$ at the constant steady state $\CSS$:
\begin{equation}
\begin{aligned}
\label{eq:definition_submatrices}
    J_1 := \nabla_{\un} &\fn(\CSS)\in\mathbb{R}^{\mn\times\mn},\ J_2 := \nabla_{\us} \fs(\CSS)\in\mathbb{R}^{\ms\times\ms},\ J_3 := \nabla_{\uf} \ffa(\CSS)\in\mathbb{R}^{\mfa\times\mfa},\\
    J_{12} &:= \nabla_{(\un, \us)} (\fn,\fs)(\CSS) = \begin{pmatrix}
        \nabla_{\un} \fn(\CSS) & \nabla_{\us} \fn(\CSS) \\
        \nabla_{\un} \fs(\CSS) & \nabla_{\us} \fs(\CSS)
    \end{pmatrix} \in\mathbb{R}^{(\mn + \ms)\times(\mn + \ms)},\\
    J_{13} &:= \nabla_{(\un, \uf)} (\fn,\ffa)(\CSS) = \begin{pmatrix}
        \nabla_{\un} \fn(\CSS) & \nabla_{\uf} \fn(\CSS) \\
        \nabla_{\un} \ffa(\CSS) & \nabla_{\uf} \ffa(\CSS)
    \end{pmatrix} \in\mathbb{R}^{(\mn + \mfa)\times(\mn + \mfa)}, \\
    J_{23} &:= \nabla_{(\us, \uf)} (\fs,\ffa)(\CSS) = \begin{pmatrix}
        \nabla_{\us} \fs(\CSS) & \nabla_{\uf} \fs(\CSS) \\
        \nabla_{\us} \ffa(\CSS) & \nabla_{\uf} \ffa(\CSS)
    \end{pmatrix} \in\mathbb{R}^{(\ms + \mfa)\times(\ms + \mfa)}.
\end{aligned}
\end{equation}

\section{Mechanisms of diffusion-driven instability}\label{sec:general_theory}

In this section we investigate mechanisms leading to diffusion-driven instability in general diffusive-nondiffusive systems. We begin by identifying subsystems whose instability is sufficient to induce DDI. These mechanisms already capture a broad class of biologically relevant models while remaining easy to verify in applications. More intricate mechanisms, which may arise in higher-dimensional systems, are discussed later in Remark~\ref{rem:DimN}.

\subsection{General conditions for DDI}\label{sec:DDI_general_conditions}
We begin with a sufficient condition for DDI induced by instability of the nondiffusive subsystem,  extending the results previously established for systems with a single diffusive component  \cite{Cygan2022}. 
\begin{proposition}[DDI induced by the nondiffusive subsystem]\label{prop:DDI_autocatalysis}
Let $\CSS$ be a constant steady state satisfying $s(J) < 0$ and $s(J_1) > 0$. 
Then~$\CSS$ exhibits DDI for all diffusion matrices $\Ds, \Df > 0$. 
Furthermore, any bounded and continuous spatially heterogeneous stationary solution of system~\eqref{eq:full_system} intersecting $\CSS$ is unstable.
\end{proposition}

\begin{proof}
Since $s(J) < 0$, the steady state~$\CSS$ is stable with respect to spatially homogeneous perturbations. 
On the other hand, by~\eqref{eq:Spectrum_of_general_L} we have
$
s(\L) \geqslant s(J_1) > 0
$
for all $\Ds, \Df > 0$ and $\CSS$ is unstable with respect to heterogeneous perturbations.

Moreover, if $\NSS$ intersects $\CSS$ at some point $x_0 \in \Omega$, then
\[
s(\tilde{\L}) \geqslant s(\nabla_{\un}\fn(\NSS)) \geqslant s(\nabla_{\un}\fn(\CSS)) = s(J_1) > 0
\]
for the linearization $\tilde{\L}$ at $\NSS$, which shows instability of $\NSS$.
The above estimate of the spectrum of the multiplication operator is based on the spectral characterization
\begin{align}\label{eq:spectrum_mult_operator}
\sigma\left( \nabla_{\un}\fn(\NSS)\right) = \overline{\bigcup_{x\in\Omega\setminus N} \sigma \left( \nabla_{\un}\fn(\NSS(x))\right)},
\end{align}
where $N\subset \Omega$ is of zero-measure \cite{KMMü23} and the fact that $\NSS$ is continuous.
\end{proof}

\begin{remark}
The result does not require the presence of all three component types: nondiffusive, slow-diffusive, and fast-diffusive in the system \eqref{eq:full_system}. 
For Proposition~\ref{prop:DDI_autocatalysis}, it suffices to have at least one nondiffusive component and at least one diffusive component, i.e., $\mn \geq 1$ and $\ms + \mfa \geq 1$.
\end{remark}

The mechanism of Proposition \ref{prop:DDI_autocatalysis} induces DDI independently of the diffusion coefficients, but at the same time rules out stable Turing patterns. This naturally raises the question whether DDI can instead originate from a coupled subsystem involving both nondiffusive and diffusive variables. 
To this end, we consider a three-scale framework distinguishing nondiffusive, slow-diffusive, and fast-diffusive subsystems and investigate whether instability can arise from a coupled block $J_{12}$ involving both the $\mn$ nondiffusive and the $\ms$ slow-diffusive variables. Writing the linearized operator $\L$ explicitly depending on the diffusion coefficients as $\L_{\Ds,\Df}$ results in
\begin{theorem}\label{thm:transfer_DDI_positive_D1}
    If $s(\L_{0,\hat{D}})>0$ for some fast diffusion matrix $\hat{D}>0$, then there exists $\varepsilon>0$ such that $s(\L_{\Ds,\hat{D}})>0$ for all slow diffusion matrices $0\leqslant \Ds<\varepsilon$.
\end{theorem}

\begin{proof}
    Consider fixed diffusion coefficients $\hat{D} > 0$ such that $s(\L_{0,\hat{D}}) > 0$.
    First, we show that this implies the existence of an eigenvalue of $\L_{0,\hat{D}}$ with strictly positive real part.
    If such an eigenvalue exists directly, we proceed with the continuity argument.
    Suppose there exists no eigenvalue with positive real part.
    Hence, by~\eqref{eq:Spectrum_of_general_L} applied with $(u,v)$ as nondiffusive components, there exists a spectral element $\lambda$ with $\Re \lambda > 0$, which belongs to $\sigma(J_{12})$.
    Then, by \cite[Section 2.1.2]{Klika12}, there exists a sequence of eigenvalues $\mu_j \in \sigma(J - \lambda_j D)$ such that $\mu_j \to \lambda$ as $j \to \infty$, with $\Re \mu_j > 0$ for sufficiently large $j$.
    Since $\mu_j$ is also an eigenvalue of $\L_{0,\hat{D}}$ for all $j\in\NN$ by \eqref{eq:eigenvalues_at_constant_steady_state}, this yields a contradiction.
    
    Next, we use the continuity of eigenvalues with respect to diffusion parameters.
    The point spectrum of $\L_{\Ds, \Df}$ is characterized by roots of the characteristic polynomial
    \begin{align}
        P(\lambda; \lambda_j, \Ds, \Df) = \det(\lambda I + \lambda_j D - J),
    \end{align}
    see \eqref{eq:eigenvalues_at_constant_steady_state}.
    Since polynomial roots depend continuously on their coefficients \cite{CuckerCorbalan}, eigenvalues vary continuously with $\Ds$ and $\Df$.
    Because there exists $k \in \mathbb{N}_0$ such that $P(\lambda; \lambda_k, 0, \hat{D})$ has a root with positive real part, 
    we can choose $\varepsilon > 0$ small enough so that for all $0 \leqslant \Ds < \varepsilon$, $P(\lambda; \lambda_k, \Ds, \hat{D})$ retains a root with positive real part yielding $s(\L_{\Ds, \hat{D}}) > 0$.
\end{proof}

\begin{remark}[DDI induced by a mixed subsystem]
    Sufficient conditions for DDI follow directly from Theorem~\ref{thm:transfer_DDI_positive_D1}. In particular, if the Jacobian matrix $J$ is stable while the principal submatrix $J_{12}$ is unstable, then the constant steady state $\CSS$ exhibits DDI for any fixed $\Df > 0$ and sufficiently small $\Ds > 0$. This shows that instability of the mixed subsystem $J_{12}$ alone is sufficient to induce DDI without necessarily destabilizing Turing patterns.

    Note that this result also applies to classical reaction--diffusion systems, where all components diffuse. The minimal assumptions are $\ms \geqslant 1$ and $\mfa \geqslant 1$.
\end{remark}

\begin{remark}\label{rem:s(L)>0}
    Combining the results of Theorem~\ref{thm:transfer_DDI_positive_D1} and equation~\eqref{eq:eigenvalues_at_constant_steady_state}, we conclude that
    \begin{align}
        s(\L) > 0 \quad \Longleftrightarrow \quad \exists\, j \in \mathbb{N}_0 : \; s(-\lambda_j D + J) > 0.
    \end{align}
\end{remark}

Finally, destabilization may also originate from other subsystems. While the previous results identify $J_1$ and $J_{12}$ as natural candidates for
inducing DDI, interactions within the block
$J_{23}$, corresponding to slow and fast diffusion, may likewise generate
instability for suitable diffusion rates (see
Example~\ref{ex:DDI_from_PDE_subsystem}). To better understand the different mechanisms, it is useful
to characterize situations in which DDI is
strictly excluded. According to the spectral decomposition in
equations \eqref{eq:Spectrum_of_general_L} and
\eqref{eq:eigenvalues_at_constant_steady_state}, diffusion cannot
destabilize the steady state if the following two conditions hold
simultaneously:
\begin{enumerate}[label=(\roman*),noitemsep,topsep=0pt]
    \item the nondiffusive submatrix $J_1$ is stable,
    \item the matrix $J-D$ remains stable for every diagonal matrix $D \geqslant 0$.
\end{enumerate}
Matrices satisfying the second condition are referred to as
\emph{strongly stable}. Using the concept of Volterra--Lyapunov
stability for matrices \cite{Cross}, one can formulate explicit criteria ensuring that
conditions (i) and (ii) hold, and therefore exclude DDI.

\begin{theorem}[Volterra--Lyapunov stability criterion]\label{thm:NoDDI}
    If $J$ is Volterra--Lyapunov stable, i.e.\ there exists a diagonal matrix $M>0$ such that $JM+MJ^T$ is a negative definite matrix, then the constant steady state $\CSS$ is stable for all diffusion coefficients~$\Ds, \Df > 0$.
\end{theorem}

\begin{proof}
    If $JM + MJ^T<0$ for a diagonal matrix $M>0$, then $(J-D)M + M(J-D)^T<0$ for every diagonal matrix $D\ge0$, and every principal submatrix of $J$ is Volterra--Lyapunov stable as well \cite{Cross}. In particular, we have $s(J-\lambda_j D)<0$ for all $j\in\NN_0$ and $s(J_1)<0$ since Volterra--Lyapunov stable matrices are stable.
    As $j\to\infty$, each eigenvalue of $J-\lambda_j D$ either has real part tending to $-\infty$ or converges to $\sigma(J_1)$ \cite[Section 2.1.2]{Klika12}.
    Together with $s(J_1)<0$, this gives $\sup_{j\in\NN_0} s(J-\lambda_j D)<0$, and hence $s(\L)<0$ by~\eqref{eq:Spectrum_of_general_L} and~\eqref{eq:eigenvalues_at_constant_steady_state} for all diffusion coefficients $\Ds, \Df>0$.
\end{proof}

Let us summarize the results of this section. The constant steady state $\CSS$ exhibits DDI if $s(J)<0$ and one of the following conditions holds:
\begin{equation}
\begin{aligned}
\label{eq:TwoConditionsDDI}
(i) &\quad s(J_1)>0, && D_{\us} \text{ arbitrary,} && D_{\uf} \text{ arbitrary,} \\
(ii) &\quad s(J_{12})>0, && D_{\us} \text{ sufficiently small,} && D_{\uf} \text{ fixed}, \\
(iii) &\quad s(J_{13})>0, && D_{\us} \text{ fixed,} && D_{\uf} \text{ sufficiently small.}
\end{aligned}
\end{equation}
For completeness, we also include the case $s(J_{13})>0$, which
corresponds to the mirrored situation obtained by interchanging the
roles of the slow- and fast-diffusive components $v$ and $w$, although
this case does not fall directly within our framework.
In the case $(i)$, all Turing patterns near $\CSS$ are unstable, whereas in case $(ii)$ and $(iii)$ it is possible to observe stable Turing patterns.

\begin{remark}\label{rem:DimN}
For systems of dimension \(m \leqslant 3\), strong stability is equivalent to the stability of \(J\) and all its principal submatrices \cite[Theorems~3--4]{Cross}.
Hence DDI can occur only if at least one principal submatrix of \(J\) is unstable. For \(m>3\), this equivalence fails: stability of all principal submatrices does not imply strong stability \cite[Example~3.1]{Satnoianu}, and DDI may arise even when every lower-dimensional subsystem is stable.
In this case, DDI emerges from the collective interplay of all components, rather than from any single low-dimensional block.
\end{remark}

Remark \ref{rem:DimN} shows that, although $J_1$ and $J_{12}$ provide
transparent mechanisms for DDI,
destabilization in reaction--diffusion--ODE systems cannot in general
be reduced to low-dimensional subsystems. In particular, for systems
with $m>3$, a complete characterization of all mechanisms leading to
DDI becomes complicated. For $m=3$, however, the situation can be analyzed in greater detail.
In this case one can systematically investigate mechanisms inducing
DDI beyond the instability of $J_1$ and $J_{12}$. Accordingly, we classify all possible sources of
DDI in the minimal three-component setting
consisting of one nondiffusive, one slow-diffusive, and one
fast-diffusive variable.

\subsection{A minimal system coupling nondiffusive, slow-diffusive and fast-diffusive components}\label{sec:special_case}

To fully characterize which subsystems can induce DDI, we employ the Routh--Hurwitz criterion. For three-component systems, the stability conditions can be expressed explicitly in this
framework. Following \cite{Sakamoto12}, let
\begin{align}
    P(\lambda) = \lambda^3 + p_1 \lambda^2 + p_2 \lambda + p_3
\end{align}
be the characteristic polynomial of $J$, with
\begin{align}\label{eq:p1_p2_p3}
p_1 &= -\tr J, &
p_2 &= \det J_{12} + \det J_{23} + \det J_{13}, &
p_3 &= -\det J.
\end{align}
The Routh--Hurwitz criterion establishes that $J$ is stable (i.e., all roots of $P(\lambda)$ lie in the open left half-plane) if and only if
\begin{align*}
    p_1 > 0, \quad p_3 > 0, \quad p_1p_2 - p_3 > 0,
\end{align*}
which also ensures $p_2 > 0$. We assume stability of $J$, a necessary condition for DDI.

DDI requires $\CSS$ to be unstable under spatially heterogeneous perturbations, i.e., $s(\L) > 0$. 
By Remark \ref{rem:s(L)>0}, this is equivalent to the existence of $j\in \mathbb{N}_0$ such that $s(- \lambda_j D + J) > 0$, where $\lambda_j$ denotes $j$-th eigenvalue of Laplacian $-\Delta$. 

We consider the characteristic polynomial of $-\mu D + J$ for $\mu\geqslant 0$, i.e.,
\begin{align}\label{eq:CharactPoly_-muD+J}
    P(\lambda; \mu) = \det(\lambda I + \mu D - J) = \lambda^3 + p_1(\mu) \lambda^2 + p_2(\mu) \lambda + p_3(\mu),
\end{align}
with coefficients
\begin{align*}
    p_1(\mu) &= \tr(\mu D - J),\\
    p_2(\mu) &= \det(J_{12} - \mu D_{12}) + \det(J_{23} - \mu D_{23}) + \det(J_{13} - \mu D_{13}),\\
    p_3(\mu) &= \det(\mu D - J),
\end{align*}
where $D_{12}=\diag(0,\Ds)$, $D_{13}=\diag(0,\Df)$ and $D_{23}=\diag(\Ds,\Df)$ denote the corresponding principal submatrices of $D$.
The assumption that $J$ is stable implies $p_1(\mu) > 0$ for all $\mu \geqslant 0$ and hence by the Routh--Hurwitz criterion, instability of $\L$ occurs if either
\begin{align}\label{eq:Routh-Hurwitz_criterion}
    p_3(\mu) < 0 \quad\text{or}\quad p_1(\mu)p_2(\mu) - p_3(\mu) < 0
\end{align}
for some $\mu=\lambda_j$ eigenvalue of $-\Delta$. 

\begin{remark}\label{rem:Turing-Hopf}
    Assuming that $J_1<0$, $\det J_{12} > 0$ and $\det J_{13} >0$, one verifies that
    \begin{align*}
        p_3(\mu) = p_3 + [(\det J_{13})\Ds + (\det J_{12})\Df] \mu - J_1 \Ds \Df \mu^2
        > 0 \quad \text{for all} \quad \mu \geqslant 0,
    \end{align*}
    which excludes the existence of zero eigenvalues of $\L$.
    This means if we have DDI assuming these inequalities, we have Turing instability caused by the destabilization of a pair of complex conjugated eigenvalues, often called Turing--Hopf bifurcation, wave instability or W-instability.

    We know that for example $\tr J_{12} > 0$ is sufficient for DDI (for $J$ stable and $D_v$ sufficiently small) and produces a Turing-Hopf bifurcation under the assumptions $J_1<0$, $\det J_{12} > 0$, $\det J_{13}>0$ of this remark.
\end{remark}

We now determine which subsystems of $J$ must be unstable for the
steady state $\CSS$ to exhibit DDI. In particular, neither $J_2>0$ nor $J_3>0$ can alone induce DDI.
\begin{theorem}\label{prop:Discussion_possibil_DDI_3_equs}
If $\CSS$ exhibits DDI, then $J$ is stable and at least one of the subsystems $J_1$, $J_{12}$, $J_{13}$, or $J_{23}$ is unstable.
\end{theorem}
\begin{proof}
By \cite[Theorem~1.2]{Sakamoto12}, if $J$ and all its one- and two-component
subsystems are stable, then the steady state remains stable for all
$\Ds,\Df>0$. Equivalently, when $J_1$, $J_{12}$, $J_{13}$, $J_{23}$, and $J$ are all stable,
one verifies that $p_3(\mu)>0$ and $p_1(\mu)p_2(\mu)-p_3(\mu)>0$ for all
$\mu\geqslant 0$. 
\end{proof}

From~\eqref{eq:TwoConditionsDDI} we know that instability of $J_1$,
$J_{12}$, or $J_{13}$ is sufficient to induce DDI, in accordance with the classical LALI (local activation,
long-range inhibition) principle.
The case in which instability originates solely from $J_{23}$ is therefore of independent interest.
First, this situation lies outside the classical Turing paradigm,
which assumes an unstable slow- or nondiffusing activator coupled
to a stable fast-diffusing inhibitor.
Second it will always cause a Turing--Hopf bifurcation, i.e.\ a Hopf bifurcation purely induced by the Turing mechanism, see Remark \ref{rem:Turing-Hopf}. 
Third, the instability of $J_{23}$ can induce DDI but only for suitable diffusion coefficients and is in general not sufficient, see Example \ref{ex:DDI_from_PDE_subsystem}.
However, it is easy to check that $\det J_{23} < 0$ is crucial to obtain DDI as $\tr J_{23} > 0$ alone cannot change the sign of $p_1(\mu)p_2(\mu)-p_3(\mu)$.

In applications it is common to apply a quasi-steady-state
approximation (QSSA) to eliminate a nondiffusing component, under the
assumption that this reduction preserves the qualitative dynamics of
the system \cite{Haerting17,Zhang2024}. In the case of three equations, QSSA results in a
classical two-component reaction--diffusion system, in which DDI requires that the slow-diffusing component satisfies the autocatalysis condition.
However, the latter requires $\det J_{12} < 0$, which can be shown for $J_1 < 0$ using the implicit function theorem.  
Consequently, if DDI in
the full model is induced by instability of $J_{23}$ or by the
condition $\tr J_{12} > 0$, then it disappears after
applying the QSSA, see Example \ref{ex:DDI_from_PDE_subsystem}.
In such case, the Turing mechanism induces a Hopf bifurcation (see Remark~\ref{rem:Turing-Hopf}), a phenomenon that cannot occur in a system consisting of only two reaction--diffusion equations.
For a more detailed discussion under which conditions DDI is preserved after applying QSSA see \cite{Smith}.

\begin{example}[DDI resulting from instability of $J_{23}$]\label{ex:DDI_from_PDE_subsystem}
    We choose
    \begin{align*}
        J = \begin{pmatrix}-1&9&1.5\\-9&-1&5\\-2&3.5&-1\end{pmatrix}.
    \end{align*}
    A direct computation yields
    \begin{align*}
        s(J) < 0, \quad J_1<0, \quad J_2<0, \quad J_3 < 0, \quad s(J_{12})<0, \quad s(J_{13})<0, \quad s(J_{23}) > 0
    \end{align*}
    and
    \begin{align*}
        p_1(\mu)p_2(\mu) - p_3(\mu) &= 3.75 + (-6.5 \Df + 71.5 \Ds) \mu\\
        &\quad+ (2\Ds^2 + 2\Df^2 + 6\Ds\Df) \mu^2 + (\Ds^2\Df + \Ds\Df^2) \mu^3.
    \end{align*}
    For instance, choosing $\Ds = 0.001$ and $\Df = 1$ gives
    \begin{align*}
        p_1(\mu)p_2(\mu) - p_3(\mu) < 0 \text{ for } \mu \in (0.767, 2.433).
    \end{align*}
    On $\Omega = (0,\pi)$ with $\lambda_1 = 1 \in (0.767,\, 2.433)$, we obtain $s(\L) > 0$, so the system exhibits DDI.
    This example also works for a fully diffusive system since instability persists if $u$ is diffusing with sufficiently small diffusion coefficient by the same continuity argument as in Theorem~\ref{thm:transfer_DDI_positive_D1}.
    
    Interestingly, if the system is reduced to its two diffusive components, the steady state is stable and no DDI occurs.
    Indeed, consider the linear system
    \begin{align*}
        \partial_t \begin{pmatrix}\un\\\us\\\uf\end{pmatrix} = \begin{pmatrix}0\\\Ds \Delta \us\\\Df\Delta \uf\end{pmatrix} + \begin{pmatrix}-1&9&1.5\\-9&-1&5\\-2&3.5&-1\end{pmatrix} \begin{pmatrix}\un\\\us\\\uf\end{pmatrix}.
    \end{align*}
    Applying the quasi-steady-state approximation (QSSA) to eliminate $\un$ yields the reduced Jacobian
    \begin{align*}
        \bar{J} = \begin{pmatrix}-82&-8.5\\-14.5&-4\end{pmatrix}, \quad \tr\bar{J} = -86 < 0, \quad \det\bar{J} = 204.75 > 0.
    \end{align*}
Since $\bar{J}$ and its principal submatrices are stable, the reduced system is stable for all choices of diffusion coefficients $\Ds, \Df > 0$.
\end{example}

\subsection{Large fast diffusion and domain rescaling}

The conditions derived in Section \ref{sec:DDI_general_conditions} describe DDI in terms of the diffusion coefficients by fixing the diffusion coefficients of the fast-diffusing subsystem and taking the remaining diffusion coefficients sufficiently small.
An equivalent formulation can be obtained by fixing the slow-diffusion coefficients and changing the fast-diffusion coefficients. However, this would additionally require a rescaling of the spatial domain.

A simple rescaling argument yields the
following counterpart of Theorem~\ref{thm:transfer_DDI_positive_D1}.

\begin{corollary}
\label{cor:transfer_stab_incr_domain}
Let $\Ds > 0$ be fixed and suppose $s(\L_{0,\hat{D}}) > 0$ for some $\hat{D} > 0$.
Then, there exists $d_0 \geqslant 1$ such that, for every $d \geqslant d_0$,
\begin{align}
s(\L_{\Ds,\,d\hat{D}}) > 0
\quad\text{on the rescaled domain}\quad \sqrt{d}\,\Omega.
\end{align}
\end{corollary}

\begin{proof}
The claim follows by rescaling the system, applying Theorem~\ref{thm:transfer_DDI_positive_D1}, and then transforming back to the original variables.  
Since Theorem~\ref{thm:transfer_DDI_positive_D1} is stated for a fixed domain, the rescaling may result in a larger spatial domain $L\Omega$ with $L \geqslant 1$.

Assume $s(\L_{0,\hat{D}}) > 0$ for some $\hat{D} > 0$ and fix $\Ds = \mathrm{diag}(d_1,\dots,d_{\ms})$.  
By Theorem~\ref{thm:transfer_DDI_positive_D1}, there exists $\varepsilon > 0$ such that  
\[
s(\L_{D,\,\hat{D}}) > 0 \quad \text{for all} \quad 0 \leqslant D < \varepsilon .
\]
Choose $d_0 \geqslant 1$ so that $d_i / d_0 < \varepsilon$ for every $i=1,\dots,\ms$.  
Then $s(\L_{\,\frac{1}{d}\Ds,\,\hat{D}}) > 0$ for all $d \geqslant d_0$.

Rescaling space by the factor $\sqrt{d}$ transforms the domain to
\[
\tilde{\Omega} = \sqrt{d}\,\Omega =: L_d \Omega
\]
and gives $s(\L_{\Ds,\,d\hat{D}}) > 0$ on $\sqrt{d}\,\Omega$ for all $d \geqslant d_0$.
\end{proof}

\begin{remark}
\label{rem:increasing_domain_nec}
    Increasing only $\Df$ while keeping $\Ds$ fixed is, in general, insufficient to induce DDI. 
    An enlargement of the spatial domain with fixed $\Ds$ is required, since this effectively decreases $\Ds$ relative to $\Df$ and thereby restores the instability mechanism.

    In case of instability of $J_{12}$, we have $s(\L_{0,\hat{D}}) > 0$ for all $\hat{D} > 0$ and can choose $\hat{D}$ arbitrary.
    If additionally $J$ is stable, we have DDI for $\Ds$ fixed, $\Df = d\hat{D}$, $d>0$ sufficiently large, on domain $\sqrt{d}\,\Omega$.
\end{remark}

\begin{remark}
In case of a system of three equations, one can explicitly determine the domain scaling required for DDI with unstable $J_{12}$ using the Routh--Hurwitz criterion.
We assume that
\[
s(J)<0,\quad J_1 < 0,\quad  \Ds \text{ fixed}, \quad \Df \text{ arbitrarily large}.
\]
Depending on whether $\tr J_{12} > 0$, $\det J_{12} < 0$, or both, one can estimate the leading-order terms of $p_3$ and $p_1p_2 - p_3$ in $\Df$ to determine the range of $\mu$ where instability of $-\mu D + J$ occurs, and thus choose $L$ such that the first positive Laplacian eigenvalue $\lambda_1$ on $L\Omega$ lies within this range.
From this, we obtain:
\begin{itemize}
    \item Case $\tr J_{12} > 0$:

    For $\varepsilon \in \big(0, \frac{\tr J_{12}}{2\Ds}\big)$, there exists $\Df > 0$ sufficiently large such that
    \begin{align}
        p_1(\mu)p_2(\mu) - p_3(\mu) < 0 \quad\text{for}\quad \mu \in \left(\varepsilon, \frac{\tr J_{12}}{\Ds} - \varepsilon\right).
    \end{align}
    Choosing $L \geqslant 1$ large enough so that $\lambda_1 < {\tr J_{12}}/{\Ds}$ 
    yields DDI.

    \item Case $\det J_{12} < 0$:

    For $\varepsilon \in \big(0, \frac{\det J_{12}}{2 J_1 \Ds}\big)$, there exists $\Df > 0$ sufficiently large such that
    \[
        p_3(\mu) < 0 \quad\text{for}\quad \mu \in \left(\varepsilon, \frac{\det J_{12}}{J_1 \Ds} - \varepsilon\right).
    \]
    Choosing $L \geqslant 1$ so that $\lambda_1 < {\det J_{12}}/({J_1 \Ds})$ yields DDI for large $\Df$.

    \item Case $\tr J_{12} > 0$ and $\det J_{12} < 0$:  
    DDI occurs for large $\Df$ if
    \[
        \lambda_1 \in \left(0, \max\left\{\frac{\tr J_{12}}{\Ds}, \frac{\det J_{12}}{J_1 \Ds}\right\}\right).
    \]
\end{itemize}
\end{remark}

\section{Existence of far-from-equilibrium patterns}
\label{sec:existence_jump}

As discussed in the introduction, reaction--diffusion--ODE systems exhibit a richer structure than classical reaction--diffusion models due to the presence of nondiffusive components.
In particular, such systems may give rise to irregular or singular patterns, including spike-type solutions or jump discontinuities, which we refer to as far-from-equilibrium patterns.
In this section, we focus on the latter case.

We establish the existence of nonconstant stationary solutions to the reaction--diffusion--ODE system~\eqref{eq:full_system}, consisting of a nondiffusive block $\Nun$ coupled to diffusive components $(\Nus, \Nuf)$. Notably, the stationary solutions may exhibit discontinuities in $\Nun$. The corresponding stationary problem is given by
\begin{nalign}
\label{eq:StationaryProblem}
\left\{
\begin{aligned}
    0 &= \fn(\Nun,\Nus, \Nuf),\\
    0 &= \Ds \Delta \Nus + \fs(\Nun,\Nus, \Nuf), \\
    0 &= \Df \Delta \Nuf + \ffa(\Nun,\Nus, \Nuf),
\end{aligned}
\right.
\end{nalign}
where $\Ds > 0$ and $\Df > 0$ denote the diffusion coefficients of the slow and fast components, respectively.
Our goal is to construct \emph{far-from-equilibrium} patterns, that is, stationary solutions whose nondiffusive component, on a set of positive measure, follows a different solution branch of the algebraic equation $\fn(\un,\us,\uf)=0$ than the one on which a constant steady state lies. The construction is based on the approach developed in \cite{Cygan2023}, suitably adapted to the present multi-scale framework.

\subsection{Two-branch structure in the \texorpdfstring{$\Nun$}{u}-nullcline}
Let $\CSS = (\Cun,\Cus,\Cuf)\in\RR^m$ denote a constant 
solution of~\eqref{eq:StationaryProblem} with Jacobian $J$ with submatrices as in~\eqref{eq:definition_submatrices}.

\begin{assumption}
\label{ass:StatSol}
The linearization matrices satisfy
\begin{nalign}
    \det J_1 \neq 0 \qquad \text{and} \qquad \det \left(J - D\lambda_j \right)\neq 0
    \end{nalign}
    for all Neumann eigenvalues $\lambda_j$ of $-\Delta$ on $\Omega$.
\end{assumption}
This ensures that the implicit function theorem applies to $\fn(\cdot, \us, \uf) = 0$, yielding a $C^2$-map
\begin{align*}
    \varphi:\mathcal{V}\times\mathcal{W} \to \mathbb{R}^{\mn},\quad \varphi(\Cus, \Cuf) = \Cun,\quad \fn(\varphi(\us,\uf),\us,\uf) = 0,
\end{align*}
for $(\us, \uf)$ in a neighborhood $\mathcal{V}\times\mathcal{W}\subseteq \RR^{\md}$ of $(\Cus, \Cuf)$, where $\md = \ms + \mfa$.
The key ingredient for our construction is the existence of \textit{multiple distinct solution branches} for the equation $\fn(\un, \us, \uf) = 0$ with respect to the nondiffusive variable $\un$, defined on the same domain $\mathcal V \times \mathcal W$.
\begin{assumption}[Two distinct branches]
\label{ass:DiscontinuousStationaryTwoBranches}
There exist two $C^2$-maps $\varphi,\psi:\mathcal{V}\times\mathcal{W}\to\mathbb{R}^{\mn}$ with
\begin{align*}
    \fn(\varphi(\us,\uf),\us,\uf) = \fn(\psi(\us,\uf),\us,\uf) = 0
\end{align*}
on $\mathcal{V}\times\mathcal{W}$ and $\varphi(\Cus,\Cuf) = \Cun \ne \psi(\Cus,\Cuf)$.
\end{assumption}
These maps form two branches of the $\Nun$-nullcline and allow constructing stationary solutions of system~\eqref{eq:full_system} that switch between those branches.

\subsection{Localised branch switching}
We study the effect of spatially localised perturbations by allowing the reaction terms to differ on complementary subdomains of~$\Omega$. The main mechanism is that a modification of the kinetics restricted to a subset of sufficiently small measure can be treated as a perturbation of the original system.

We first establish uniform $W^{2,p}$-bounds for the diffusive subsystem.

\begin{lemma}
	\label{DisLemma1} 
    Fix $p\in [2, \infty)$. Let $ D^d\in\RR^{\md\times \md}$ be a diagonal matrix with positive entries, and let $J^d \in \RR^{\md\times \md}$ satisfy
    \begin{align}
    \det(J^d - D^d\lambda_j) \neq 0 \quad \text{for all} \quad j\in\NN_0.
    \end{align}
    Then the operator 
    \begin{align}
        \L^d := D^d\Delta + J^d \quad \text{with domain} \quad D\left(\mathcal L^d\right)=\left\{\nu\in W^{2,p}(\Omega)^{\md}:\ \partial_n \nu=0\text{ on }\partial\Omega\right\}
    \end{align} 
    is invertible, and $\| (\L^d)^{-1} \xi\|_{W^{2,p}(\Omega)} \leqslant C \|\xi\|_{L^p(\Omega)}$ 
    for each $\xi \in L^p(\Omega)^{\md}$.
\end{lemma}

\begin{proof}
Let $\nu\in D(\mathcal L^d)$, where the domain of $\mathcal L^d$ is as characterized in the statement, owing to the $C^2$-boundary of $\Omega$. By the standard elliptic estimate for the Neumann problem,
\begin{align}
\|\nu\|_{W^{2,p}(\Omega)^{m_d}}
\leqslant C\Big(\left\|\mathcal L^d \nu\right\|_{L^p(\Omega)^{m_d}}+\|\nu\|_{L^p(\Omega)^{m_d}}\Big).
\end{align}
The spectral characterization given by formula \eqref{eq:eigenvalues_at_constant_steady_state} with $\mn = 0$ yields $0\notin \sigma(\L^d)$ on $L^2(\Omega)^{\md}$.
Therefore, $\mathcal L^d:D(\mathcal L^d)\to L^p(\Omega)^{\md}$ is invertible and the inverse is bounded for each $p \in [2,\infty)$. The proof can be found \textit{e.g.} in \cite[Corollary 2.4]{KMMü23} for $p>n/2$ but it can be easily extended for $p\geqslant 2$.
Therefore, setting $\nu = (\L^{d})^{-1} \xi$ we obtain
\begin{align}
\left\|\left(\L^d\right)^{-1}\xi\right\|_{W^{2,p}(\Omega)^{m_d}}
\leqslant C\Big(\|\mathcal \xi\|_{L^p(\Omega)^{m_d}}+\left\|\left(\L^d\right)^{-1} \xi\right\|_{L^p(\Omega)^{m_d}}\Big) \leqslant \tilde{C}\|\mathcal \xi\|_{L^p(\Omega)^{m_d}}.
\end{align}
\end{proof}

Now we state a perturbation lemma allowing different reaction terms on two complementary subdomains.

\begin{lemma}
	\label{DisTheorem1}
	Let $p>n/2$, $p\geqslant 2$ and $q_1, q_2 \in C_b^2(\RR^{\md},\RR^{\md})$ {\rm (}\textit{i.e.}\ all derivatives up to order two are bounded{\rm )}. 
    Assume that for some $\overline{\nu} \in \RR^{\md}$ and $J^d:=\nabla q_1(\overline{\nu})$ we have
	\begin{nalign}
		q_1(\overline{\nu}) = 0 \quad \text{and} \quad \det\left(J^d - D^d \lambda_j\right)\neq 0 \quad \text{for each} \ j\in\NN_0,
	\end{nalign} 
    with a diagonal matrix $D^d>0$.\\
    Then, there exists $\varepsilon_0>0$ such that for all $\varepsilon\in(0,\varepsilon_0)$ there exists $\delta>0$ such that for any open subset $\Omega_1\subset \Omega$ with \mbox{$\Omega_2 = \Omega\setminus\Omega_1$} and $|\Omega_2| \leqslant \delta$, the problem
	\begin{nalign}
		\label{DisDefSingle}
		D^d\Delta \nu + q_1({\nu}) \mathbbm{1}_{\Omega_1} + q_2({\nu}) \mathbbm{1}_{\Omega_2} = 0 \quad \text{in}\ \Omega
	\end{nalign}
	admits a solution $ \nu\in W^{2,p}(\Omega)^{\md}$ with $\partial_n\nu=0$ on $\partial\Omega$ and $\|\nu-\overline{\nu}\|_{W^{2,p}(\Omega)^{m_d}}\leqslant \varepsilon$. 
\end{lemma}

\begin{proof}
The argument is a straightforward vector-valued extension of the proof of the corresponding scalar result based on the Banach fixed point theorem \cite[Lemma 3.2]{Cygan2023}. We therefore indicate only the modifications required in the present setting.

Let $\L^d = D^d\Delta + J^d$.
After replacing \(\nu\) by \(\overline{\nu}+\omega\) and using Lemma~\ref{DisLemma1}, equation \eqref{DisDefSingle} can be written in the equivalent form
\begin{align}
\omega = \mathcal T_1(\omega)+\mathcal T_2(\omega),
\end{align}
where
\begin{align}
\mathcal T_1(\omega):=\left(\mathcal L^d\right)^{-1}\Bigl(\bigl(J^d\omega-q_1(\overline{\nu}+\omega)\bigr)\mathbbm{1}_{\Omega_1}\Bigr),
\qquad
\mathcal T_2(\omega):=\left(\mathcal L^d\right)^{-1}\Bigl(\bigl(J^d\omega-q_2(\overline{\nu}+\omega)\bigr)\mathbbm{1}_{\Omega_2}\Bigr).
\end{align}
We consider the ball
\begin{align}
B_\varepsilon:=\{\omega\in W^{2,p}(\Omega)^{\md}:\|\omega\|_{W^{2,p}(\Omega)^{\md}}\leqslant \varepsilon\}.
\end{align}
Since $q_1\in C_b^2(\RR^{\md},\RR^{\md})$ and $q_1(\overline{\nu}) = 0$, Taylor's theorem with remainder yields
    \begin{align*}
        \left|J^d\omega(x) - q_1\big(\overline\nu+\omega(x)\big)\right|
        \leqslant \frac{1}{2}\sup_{z\in\RR^{\md}}\big|D^2 q_1(z)\big|\,\left|\omega(x)\right|^2
        =: C\left|\omega(x)\right|^2,
\end{align*}
where $D^2q_1$ denotes the second derivative of $q_1$.
Using the boundedness of $(\mathcal L^d)^{-1}$ from Lemma~\ref{DisLemma1} and the embedding
\(W^{2,p}(\Omega)\hookrightarrow L^\infty(\Omega)\), valid since \(p>n/2\), we obtain
\begin{align}
\|\mathcal T_1(\omega)\|_{W^{2,p}(\Omega)^{\md}}
\leqslant C\|\omega\|_{L^\infty(\Omega)^{\md}}^2
\leqslant C\varepsilon^2,
\end{align}
and similarly
\begin{align}
\|\mathcal T_1(\omega)-\mathcal T_1(\widetilde \omega)\|_{W^{2,p}(\Omega)^{\md}}
\leqslant C\varepsilon \|\omega-\widetilde \omega\|_{W^{2,p}(\Omega)^{\md}}
\end{align}
for \(\omega,\widetilde \omega\in B_\varepsilon\).

For the term on \(\Omega_2\), we only use that $q_2$ is bounded and Lipschitz continuous yielding
\begin{align}
|J^d\omega-q_2(\overline{\nu}+\omega)|\leqslant C(1+|\omega|),
\end{align}
which implies
\begin{align}
\|\mathcal T_2(\omega)\|_{W^{2,p}(\Omega)^{\md}}
\leqslant C |\Omega_2|^{1/p}(1+\|\omega\|_{L^\infty(\Omega)^{\md}})
\leqslant C |\Omega_2|^{1/p}(1+\varepsilon),
\end{align}
and
\begin{align}
\|\mathcal T_2(\omega)-\mathcal T_2(\widetilde \omega)\|_{W^{2,p}(\Omega)^{\md}}
\leqslant C |\Omega_2|^{1/p}\|\omega-\widetilde \omega\|_{W^{2,p}(\Omega)^{\md}}.
\end{align}

Hence, choosing \(\varepsilon>0\) sufficiently small and then \(\delta>0\) such that
\(|\Omega_2|\leqslant\delta\), the map \(\mathcal T:=\mathcal T_1+\mathcal T_2\) maps \(B_\varepsilon\) into itself and defines a contraction on \(B_\varepsilon\).
The Banach fixed point theorem yields a fixed point $\omega\in B_\varepsilon$, and $\nu=\overline{\nu}+\omega$ is the desired solution of~\eqref{DisDefSingle}.
\end{proof}

\subsection{Main existence theorem} Combining the geometric assumptions on the $\Nun$-nullcline with the regularity and perturbation lemmas, we now establish the existence of far-from-equilibrium stationary solutions.
The result is obtained by applying the localised branch-switching lemma to the $(\Nus,\Nuf)$-subsystem and reconstructing $\Nun$ from the chosen nullcline branch.
In contrast to the conditions for DDI discussed in the preceding sections, the diffusion coefficients do not play a role for the existence (except for the admissible size of $\Omega_2$) as long as $D_v>0$, $D_w>0$ and Assumption \ref{ass:StatSol} is satisfied.

\begin{theorem}
	\label{DisExBan}
    Consider problem \eqref{eq:StationaryProblem} with a constant solution $\CSS$ satisfying Assumption~\ref{ass:StatSol} and Assumption \ref{ass:DiscontinuousStationaryTwoBranches}.
    There exists $\delta>0$ such that for any open subset $\Omega_1\subseteq\Omega$ with $\Omega_2 = \Omega \setminus \Omega_1$ satisfying $|\Omega_2|< \delta$, problem \eqref{eq:StationaryProblem} admits a weak solution $(\Nus,\Nuf)= \big(\Nus(x),\Nuf(x)\big)$ to
		\begin{nalign}
			\label{eq:DinsontinousSolutionsVEquation}
			\Ds \Delta \Nus + \fs(\Nun, \Nus, \Nuf) = 0 \qquad \text{and} \qquad \Df \Delta \Nuf + \ffa(\Nun, \Nus, \Nuf) = 0, 
		\end{nalign}
		where
		\begin{nalign}
			\label{Uform}
			\Nun(x) = \begin{cases}
				\varphi\big(\Nus(x), \Nuf(x)\big), \quad x\in \Omega_1, \\
				\psi\big(\Nus(x), \Nuf(x)\big), \quad x\in \Omega_2,
			\end{cases}
		\end{nalign} 
		satisfies $\fn\big(\Nun(x),\Nus(x), \Nuf(x)\big) = 0$ for almost all $x\in \Omega$. 
\end{theorem}

\begin{proof}
    Let $D^d:=\diag(\Ds,\Df)$ and choose $r>0$ such that $\overline{B_{2r}(\Cus, \Cuf)} \subset \mathcal{V}\times \mathcal W\subset\RR^{\md}$, where $\md = \ms+\mfa$.
    Since $\fn$, $\fs$ and $\ffa$ are of class $C^2$, so is $\varphi$ by the implicit function theorem, and $\psi$ by Assumption~\ref{ass:DiscontinuousStationaryTwoBranches}.
    Hence the maps $(\us,\uf)\mapsto (\fs,\ffa)\big(\varphi(\us,\uf),\us,\uf\big)$ and $(\us,\uf)\mapsto (\fs,\ffa)\big(\psi(\us,\uf),\us,\uf\big)$ are bounded together with their first and second derivatives on the compact set $\overline{B_{2r}(\Cus, \Cuf)}$.
    Let $\chi\in C_c^\infty\big(B_{2r}(\Cus, \Cuf)\big)$ with $0\leqslant\chi\leqslant1$ and $\chi\equiv1$ on $B_r(\Cus, \Cuf)$, and define
    \begin{align}\label{eq:def_q1_q2}
        q_1(\us,\uf) = \chi(\us,\uf)\begin{pmatrix}    
        \fs\left( \varphi(\us,\uf), \us,\uf \right)    \\
        \ffa\left( \varphi(\us,\uf), \us,\uf \right)
        \end{pmatrix} \qquad \text{and} \qquad 
        q_2(\us,\uf) = \chi(\us,\uf)\begin{pmatrix}    
        \fs\left( \psi(\us,\uf), \us,\uf \right)    \\
        \ffa\left( \psi(\us,\uf), \us,\uf \right)
        \end{pmatrix}
    \end{align}
    for $(\us,\uf)\in B_{2r}(\Cus, \Cuf)$, extended by zero to $\RR^{\md}$.
    Then $q_1, q_2 \in C_b^2(\RR^{\md},\RR^{\md})$ and $q_1(\Cus, \Cuf) = 0$.
    Differentiating the identity $\fn(\varphi(\us,\uf),\us,\uf) = 0$ at $(\Cus, \Cuf)$ gives
    \begin{align}\label{eq:deriv_tilde_phi}
        \nabla_{(\us,\uf)} \varphi(\Cus, \Cuf) = -J_1^{-1} \begin{pmatrix}
        \nabla_{\us}\fn\left(\CSS\right) & \nabla_{\uf}\fn\left(\CSS\right)
        \end{pmatrix},
    \end{align}
    and hence, since $\chi\equiv1$ near $(\Cus, \Cuf)$, the chain rule shows that $\nabla_{(\us,\uf)} q_1(\Cus, \Cuf)$ is the Schur complement of $J_1$ in $J$.
    Since $D=\diag(0,D^d)$, Assumption~\ref{ass:StatSol} and Schur's formula for determinants of block matrices yield
    \begin{nalign}
		\det \big(\nabla_{(\us,\uf)} q_1(\Cus, \Cuf) - D^d\lambda_j\big) =        
        \frac{1}{\det J_1}\det\left(J - D\lambda_j\right)\neq 0
	\end{nalign}
    for all $j\in\NN_0$, so $q_1$ and $q_2$ satisfy the assumptions of Lemma~\ref{DisTheorem1}.
    
    Fix $p\in[2,\infty)$ with $p>n/2$, and let $C_S$ denote the constant of the Sobolev embedding $W^{2,p}(\Omega) \subseteq L^\infty(\Omega)$.
    For $\varepsilon\in(0,\varepsilon_0)$ with $C_S\varepsilon<r$, Lemma~\ref{DisTheorem1} provides $\delta>0$ and, whenever $|\Omega_2|\leqslant\delta$, a solution $(\Nus,\Nuf)\in W^{2,p}(\Omega)^{\md}$ of problem~\eqref{DisDefSingle} with $\| \Nus-\Cus\|_{W^{2,p}(\Omega)^{\ms}} + \| \Nuf-\Cuf\|_{W^{2,p}(\Omega)^{\mfa}} \leqslant \varepsilon$.
    Then $(\Nus(x),\Nuf(x))\in B_r(\Cus,\Cuf)$ for all $x\in\Omega$, where $q_1$ and $q_2$ coincide with $(\fs,\ffa)(\varphi(\cdot),\cdot)$ and $(\fs,\ffa)(\psi(\cdot),\cdot)$, respectively.
    Defining $\Nun$ by~\eqref{Uform}, the triple $(\Nun,\Nus,\Nuf)$ therefore solves problem~\eqref{eq:StationaryProblem}.
\end{proof}

\begin{remark}[Proximity to $\CSS$]\label{rem:Jump_pattern_far_from_CSS}
    The stationary solution $(\Nun, \Nus,\Nuf)$ constructed in Theorem \ref{DisExBan} stays close to the points 
    $\big(\Cun, \Cus,\Cuf\big)$  and 
    $\big(\psi (\Cus,\Cuf), \Cus,\Cuf\big)$
    in the following sense.
    There exists $\varepsilon_0>0$ such that, for every $\varepsilon \in (0,\varepsilon_0)$, we find $\delta(\varepsilon)>0$ such that, whenever $|\Omega_2|<\delta(\varepsilon)$, the solution constructed in Theorem \ref{DisExBan} satisfies
    \begin{gather}
        \label{DisVar}
        \|\Nus - \Cus \|_{L^\infty(\Omega)^{\ms}}  < \varepsilon, \qquad \|\Nuf - \Cuf \|_{L^\infty(\Omega)^{\mfa}} < \varepsilon, \quad \text{and} \\ 
        \|\Nun -  \Cun \|_{L^\infty(\Omega_1)^{\mn}} + \|\Nun - \psi (\Cus, \Cuf) \|_{L^\infty(\Omega_2)^{\mn}}  < C\varepsilon.
    \end{gather}  
    If $|\Omega_2|>0$, we refer to such a solution as a \emph{far-from-equilibrium pattern}, since its nondiffusive component follows the second branch $\psi$ on $\Omega_2$ and therefore stays away from the equilibrium $(\Cun, \Cus, \Cuf)$ there, at a distance governed by the separation of the two branches of $f=0$.
\end{remark}

\begin{remark}[Continuity and jumps]
\label{rem:Conti}
By Sobolev embedding, the diffusive components $\Nus$ and $\Nuf$ belong to $C(\overline{\Omega})$ and are therefore continuous.  
In contrast, the nondiffusive component $\Nun$, defined by~\eqref{Uform}, may exhibit jump discontinuities at points $x\in\partial\Omega_1\cap\Omega$ with $\varphi(\Nus(x),\Nuf(x)) \neq \psi(\Nus(x),\Nuf(x))$.  
Since $\varphi(\Cus,\Cuf)\neq\psi(\Cus,\Cuf)$ and $\varphi$, $\psi$ are continuous, this holds at every such point when the parameter $\varepsilon>0$ in~\eqref{DisVar} is chosen sufficiently small.
\end{remark}

\begin{remark}[Multiplicity and stability]\label{rem:Mult_stab}
The subset $\Omega_2\subset\Omega$ can be chosen arbitrarily, provided its measure is sufficiently small, yielding infinitely many distinct stationary solutions. In practice, $\Omega_2$ does not need to be very small, as illustrated in Section~\ref{sec:Application}. The stability of these patterns is a separate issue, which may depend on the geometry of $\Omega_2$ and on the diffusion coefficients, and is in general not guaranteed.

For a single diffusive component ($\ms+\mfa=1$), \cite[Theorem~2.11]{Cygan2023} gives an explicit sufficient condition: if the nondiffusive submatrix $\nabla_uf$ is stable when evaluated at both $\CSS$ and $(\psi(\Cus,\Cuf),\Cus,\Cuf)$, $s(J)<0$ at $\CSS$, the diffusion coefficient is sufficiently large, and $\varepsilon,|\Omega_2|>0$ are sufficiently small, then the pattern constructed in Theorem~\ref{DisExBan} is linearly exponentially stable in $L^p(\Omega)^{\mn+1}$. By \cite{KMMü23}, the linear stability implies nonlinear stability in $L^\infty(\Omega)^{\mn+1}$ and Lyapunov stability in $L^\infty(\Omega)^{\mn}\times C(\overline\Omega)$.

Conversely, suppose the submatrix $\nabla_{\un} f$, evaluated at either $\CSS$ or at $(\psi(\Cus,\Cuf),\Cus,\Cuf)$, has an eigenvalue with positive real part. Then, for sufficiently small $\varepsilon>0$, the stationary solutions constructed in Theorem~\ref{DisExBan} (with $|\Omega_2|>0$ in the latter case) are linearly unstable. This follows from Remark~\ref{rem:Jump_pattern_far_from_CSS} together with the spectral characterization in~\eqref{eq:Spectrum_of_general_L} and~\eqref{eq:spectrum_mult_operator}, together with standard perturbation arguments for linear operators.

Finally, none of these patterns can be \emph{asymptotically} stable in $L^p(\Omega)^{m}$ for any finite $p$: shrinking $\Omega_2$ slightly always produces another stationary solution that is arbitrarily close in $L^p(\Omega)^{m}$ for $p<\infty$, yet stays at a fixed positive $L^\infty$-distance, bounded below by the separation between the branches $\varphi$ and $\psi$. Asymptotic stability, if it occurs at all, can therefore only occur in the $L^\infty$-topology, never in $L^p$ for finite $p$.
\end{remark}

\section{Application to a receptor-based model}\label{sec:Application}

In this section, we illustrate the theoretical framework developed above by applying it to a receptor-based model consisting of one ODE and two PDEs. The model is inspired by~\cite{Marciniak06}, and describes the dynamics of receptors (\(u\)), ligands (\(v\)), and enzymes (\(w\)) on a one-dimensional spatial domain.

Let $\Omega = (0,L)\subset\mathbb{R}$, with $L>0$. The governing equations are
\be \label{eq:toysys}
\left\{
\begin{aligned}
    \del_t u &= -\mu_1 u + m_1 uv (1 + u v)^{-1}, & \textup{for } \ x \in \bar\Omega, \ t > 0,\\
    \del_t v &= D_v \Delta v -\mu_2 v + m_2 uv (1 + uv)^{-1} - vw, & \textup{for } \ x \in \Omega, \ t > 0,\\
    \del_t w &= D_w \Delta w -\mu_3 w + m_3 uv (1 + uv)^{-1}, & \textup{for } \ x \in \Omega, \ t > 0,
\end{aligned}
\right.
\ee
supplemented with homogeneous Neumann boundary conditions for $v$ and $w$, and nonnegative initial data.
Here $\mu_i \in (0,1]$ denote decay rates, the $m_i \in [2,+\infty)$ are production rates, and $D_v, D_w > 0$ are the diffusion coefficients of ligands and enzymes respectively. The receptor population, $u$, is assumed diffusion-free.
The nonlinear production terms follow a Michaelis--Menten type saturation in the ligand--receptor binding.

Introducing $X = \left(u, v, w \right)^T$, $X_0 = \left(u_0, v_0, w_0 \right)^T$, and $D = \diag(0, D_v, D_w)$, the system is written in vector form
\begin{align}
\del_t X = D \Delta X + F \bigl( X \bigr), \quad \text{for} \ x \in \Omega \ \text{and} \ t > 0, 
\end{align}
with nonlinearities
\[F(X)  = \bpm -\mu_1 u + m_1 uv (1 + uv)^{-1} \\ -\mu_2 v + m_2 uv (1 + uv)^{-1} - vw \\ -\mu_3 w + m_3 uv (1 + uv)^{-1}\epm = \bpm f(X) \\ g(X) \\ h(X) \epm. \]
The existence, uniqueness, and boundedness of solutions are established in Appendix~\ref{app2}.

For the analysis of DDI, we require that the reduced kinetic system
\begin{align}
    \dot{ X}(t) = F \bigl( X(t) \bigr), \qquad X(0) = X_0,
\end{align}
possesses an asymptotically stable, non-negative steady state. 
To simplify the notation, we introduce the quantities
\begin{align}
\label{eq:EtaDef}
\eta_1 := \frac{m_1}{\mf}, \quad \eta_2 := \frac{m_2}{\ml}, \quad \eta_3 := \frac{m_3}{\me}, \quad \alpha := m_2 \eta_1 - \eta_3.
\end{align}

To formulate the upcoming results, we first collect the parameter constraints into the following assumption.

\begin{assumption}\label{assump:Parameters}
Let $\mu_i \in (0,1]$, $m_i\ge2$ for $i=1,2,3$, and $\eta_1, \eta_2, \eta_3, \alpha$ as defined in~\eqref{eq:EtaDef}. We assume that the following conditions hold
\begin{enumerate}
    \item\label{item:first_assumption} $\zeta > 0$, where $\zeta := \alpha - 2\mu_2,$
    \item\label{item:sec_assumption} $\Theta > 0$, where $\Theta := \zeta^2 - 4 \mu_2 (\mu_2 + \eta_3),$
\end{enumerate}
\end{assumption}

Under Assumption~\ref{assump:Parameters}, the kinetic system admits three nonnegative steady states, as stated below.

\begin{theorem}\label{thm:bistabletoysys}
Under Assumption~\ref{assump:Parameters} the associated kinetic system admits three constant steady states:
\begin{align}
\CSS_{0} := (0, 0, 0)
\quad\text{and}\quad
\CSS_\pm := (\bar u_\pm, \bar v_\pm, \bar w_\pm),
\end{align}
with
\begin{align}
\label{eq:steadystates}
\bar v_\pm= \frac{\alpha + 2\eta_3 \pm \sqrt{\Theta}}{2 \eta_1 (\mu_2 + \eta_3)},\quad
\bar u_\pm = \eta_1 - \frac{1}{\bar v_\pm},\quad
\bar w_\pm = \eta_3 - \frac{\eta_3}{\eta_1 \bar v_\pm}.
\end{align}
Here, $\CSS_{0}$ is asymptotically stable, $\CSS_-$ is positive and unstable, and $\CSS_+$ is positive. Moreover, $\CSS_+$ is asymptotically stable provided
\begin{align}\label{eq:stab_cond_X+}
m_1 \bar v_+ + m_2 \bar u_+ > \Big(\frac{\mu_1}{\mu_3} - 1\Big)\eta_3.
\end{align}
\end{theorem}

The proof of this result is given in Appendix \ref{app1}.
Assume from now on that Assumption~\ref{assump:Parameters} and condition~\eqref{eq:stab_cond_X+} hold. 
Notice that \eqref{eq:stab_cond_X+} holds trivially if $\mu_1 \leqslant \mu_3$ by positivity of $\CSS_+$.

\subsection{Diffusion-driven instability}
We now investigate whether the stable positive steady state $\CSS_+$ can destabilize through DDI.
Since $\CSS_{0}$ is Volterra--Lyapunov stable (Theorem~\ref{thm:NoDDI}) and $\CSS_-$ is unconditionally unstable, only $\CSS_+$ is relevant for DDI.

Let $J := J(\CSS_+)$ be the Jacobian matrix of the nonlinearities $F$ evaluated at $\CSS_+$. A transformation using the relation $F(\CSS_+) = 0$, allows us to express $J$ as follows:
\begin{align}    
J 
= \dfrac{1}{(1 + \bar u_+ \bar v_+)^2} \bpm -m_1 \bar u_+ \bar v_+^2 & m_1 \bar u_+  & 0 \\ m_2 \bar v_+ & -m_2 \bar u_+ ^2 \bar v_+ & -\bar v_+ (1 + \bar u_+ \bar v_+)^2 \\ m_3 \bar v_+ & m_3 \bar u_+  & -\me (1 + \bar u_+ \bar v_+)^2 \epm.
\end{align}

The diagonal element $J_1<0$ rules out the autocatalysis condition of Proposition~\ref{prop:DDI_autocatalysis}.
Furthermore, the $2\times 2$ submatrices $J_{13}$ and $J_{23}$ are stable (negative trace, positive determinant), leaving $J_{12}$ as the only viable source of DDI (Theorem~\ref{prop:Discussion_possibil_DDI_3_equs}).

\begin{theorem}[Instability of $J_{12}$]
\label{thm:J12unstable}
Suppose Assumption~\ref{assump:Parameters} and condition~\eqref{eq:stab_cond_X+} hold and, in addition,
\begin{align}\label{eq:cond_instab_J12}
    \eta_1 < 2\left(\frac{\eta_3}{m_2} + \frac{2}{\eta_2}\right).
\end{align}
Then $J_{12}$ has an unstable eigenvalue.
Consequently, there exists either $D_v > 0$ sufficiently small, or a pair $(D_w, L)$ with $D_w > 0$ and $L \geqslant 1$ both sufficiently large, such that $\CSS_+$ undergoes DDI.
\end{theorem}

\begin{proof}
From $\tr (J_{12}) = - \bar u_+ \bar v_+ (m_1 \bar v_+ + m_2 \bar u_+)/(1 + \bar u_+ \bar v_+)^2$, it is clear that the sub-matrix $J_{12}$ has a negative trace. Thus, the $2 \times 2$ matrix $J_{12}$ has an eigenvalue with positive real part if and only if $\det J_{12}<0$. This occurs when the parameters satisfy
\begin{align}
    \det(J_{12}) 
    = m_1 m_2 \frac{\bar u_+ \bar v_+}{\big(1+\bar u_+ \bar v_+\big)^4} 
      \left( (\bar u_+ \bar v_+)^2 - 1 \right) < 0.
\end{align}

Clearly, this inequality holds whenever $\bar u_+ \bar v_+ < 1$.  By definition of $\bar u_+$, the inequality is satisfied when $\bar v_+ < 2 / \eta_1$. Substituting the explicit expression for $\bar v_+$, we obtain
\be
\label{eq:ineq}\alpha + \sqrt{\alpha^2 - 4 \mu_2 (\alpha + \eta_3)} <  2 \bigl( \eta_3 + 2 \mu_2 \bigr).\ee
Using the definition of $\alpha$, inequality \eqref{eq:ineq} simplifies to the desired condition~\eqref{eq:cond_instab_J12}.

If \eqref{eq:cond_instab_J12} is satisfied, then $s(J_{12}) > 0$.  
By Theorem~\ref{thm:transfer_DDI_positive_D1} and Corollary~\ref{cor:transfer_stab_incr_domain}, there exists either $D_v > 0$ sufficiently small, or a pair $(D_w, L)$ with $D_w > 0$ and $L \geqslant 1$ both sufficiently large, such that the steady state $\CSS_+$ exhibits DDI.
\end{proof}

Theorem~\ref{thm:J12unstable} guarantees the existence of DDI under broad parameter regimes.
The threshold values of $D_v,D_w$, and $L$ in Theorem~\ref{thm:J12unstable} can be characterized explicitly using the Routh--Hurwitz criterion applied to the characteristic polynomial of $J - \lambda_j D$ (Theorem~\ref{lem:DDI_toysys_exact}).
See Section~\ref{sec:special_case} for the notation.
For this system, it is straightforward to verify that the Routh--Hurwitz elements $p_1(\lambda_j)$, $p_2(\lambda_j)$, and $(p_1p_2 - p_3)(\lambda_j)$ are positive for all $\lambda_j \geqslant 0$ and $D_v, D_w > 0$. 
Instability can therefore arise only if $D_v$ and $D_w$ are chosen such that, for some fixed eigenvalue $\lambda_j$ the Routh--Hurwitz element $p_3(\lambda_j) := - \det \big(J - \lambda_j D\big)$ is negative.

\begin{theorem}
\label{lem:DDI_toysys_exact}
    Let $\Omega = (0, L)$ and consider the constant steady state $\CSS_+$ of~\eqref{eq:toysys}. This steady state exhibits DDI if either of the following conditions is satisfied:
    \begin{enumerate}
        \item[(I)] $L >0$, the diffusion coefficient $D_w > 0$ is fixed and 
        \begin{align} 
        0 < D_v < \varepsilon:=\sup_{j\in\NN}\dfrac{\det(J) - \det(J_{12}) \, D_w \lambda_j}{\lambda_j \lp \det(J_{13}) - J_1 \, D_w \lambda_j \rp}, 
        \end{align}
        \item[(II)] $D_v > 0$ fixed, and
        $$ L > \pi j \sqrt{\dfrac{J_1 \, D_v }{\det(J_{12})}}, \qquad D_w >  \dfrac{\det(J) - \det(J_{13})\,D_v \lambda_j}{\lambda_j(\det(J_{12}) - J_1\,  D_v \lambda_j)} > 0 $$
        holds for at least one $j \in \mathbb N$.
    \end{enumerate} 
\end{theorem}

\begin{proof}
    To determine the precise parameter regimes for DDI, we analyze the sign of the Routh--Hurwitz element $p_3(\lambda_j) = - \det(J - \lambda_j D)$. Instability requires $p_3(\lambda_j)<0$, and we distinguish the cases of small $D_v$ and large $D_w$ accordingly.
    
    By solving $p_3(\lambda_j) < 0$ for $D_v$, that is, finding values of $D_v$ such that
    \begin{align}
    -\det(J) + \biggl( \det(J_{13})\,D_v + \det(J_{12}) \,D_w \biggr) \lambda_j + \biggl( -J_1\, D_v D_w \biggr) \lambda_j^2 < 0,
    \end{align}
    we obtain condition~\emph{(I)}.  
    The right-hand side of the resulting inequality is positive for sufficiently large $j \in \mathbb{N}$ and converges to $0$, since $J_{12}$ is unstable while $J_1$ and $J_{13}$ are always stable.

    The second condition is obtained by fixing $j \in \mathbb{N}$ and rewriting $p_3(\lambda_j) < 0$ as
    \begin{align}
    \label{eq:toy_critical_d2}
        \lambda_j \big( \det(J_{12}) - J_1\, D_v \lambda_j \big) D_w 
        < \det(J) - \det(J_{13})\,D_v \lambda_j.
    \end{align}
    Since the right-hand side of~\eqref{eq:toy_critical_d2} is negative, this inequality can only hold for large $D_w$ if
    \begin{align}
        \det(J_{12}) - J_1 D_v \lambda_j < 0.
    \end{align}
    The latter can always be achieved by increasing the domain length $L$ if necessary. 
    Solving~\eqref{eq:toy_critical_d2} for $D_w$ and the above inequality for $L$ yields the two expressions stated in condition~\emph{(II)}.
\end{proof}

\subsection{Numerical simulations}
We first verify that the parameter constraints introduced in Theorem~\ref{thm:bistabletoysys} and Theorem~\ref{thm:J12unstable} define a non-empty region. These conditions guarantee the existence of the asymptotically stable steady state $\CSS_+$ and the occurrence of DDI through $s(J_{12}) > 0$.

Let $p := (m_1, m_2, m_3, \mu_1, \mu_2, \mu_3)$ denote the vector collecting parameters of system~\eqref{eq:toysys}, and $\mathcal P := [2, +\infty)^3 \times [0, 1]^3$ the admissible parameter domain. Within $\mathcal P$ we define
\begin{align*}
\mathcal R_1 &:= \biggl\{ p \in \mathcal P : \zeta > 0, \, \text{ where } \zeta := \alpha - 2\mu_2\biggr\},\\
\mathcal R_2 &:= \biggl\{ p \in \mathcal P : \Theta > 0, \text{ where } \Theta := \zeta^2 - 4\mu_2(\mu_2 + \eta_3) \biggr\},\\ 
\mathcal R_3 &:= \biggl\{ p \in \mathcal P : \alpha + 2\eta_3 + \sqrt{\Theta} > 2 \lp \dfrac{\mu_1}{\mu_3} - 1 \rp (\mu_2 + \eta_3)\dfrac{\eta_1 \eta_3}{m_1 + m_2} \biggr\},\\
\mathcal R_4 &:= \biggl\{ p \in \mathcal P : \eta_1 < 2 \left(\dfrac{\eta_3}{m_2} + \dfrac{2}{\eta_2}\right) \biggr\}.
\end{align*}
The set $\mathcal R_3$ is equivalent to $\bar v_+ > (\mu_1/\mu_3-1)\eta_3/(m_1+m_2)$ and therefore implies the stability condition~\eqref{eq:stab_cond_X+} whenever $\bar u_+ \geqslant \bar v_+$.
A numerical exploration shows that 
\begin{align*}
    p_\star = (2.5, 9.68, 7.0, 0.95, 0.95, 0.6) \in \bigcap_{i=1}^4 \mathcal{R}_i,
\end{align*}
confirming that the feasible region is non-empty.
Moreover, at $p_\star$ we have $\bar u_+ > \bar v_+$ yielding stability of $\CSS_+$.
Since the inequalities defining $\mathcal R_1,\dots,\mathcal R_4$ are strict, an open neighborhood $\mathcal{V}_{p_\star} \subset \mathcal P$ exists in which all constraints remain satisfied (see Figure~\ref{fig:paramspace}).
The parameter set $(2.5, 9.68, 7.0, 1.0, 1.0, 0.6)$ used in the simulations below also lies in $\bigcap_{i=1}^4 \mathcal R_i$ and satisfies $\bar u_+ > \bar v_+$.

\begin{figure}[t]
    \centering
    \includegraphics[width=1.0\textwidth]{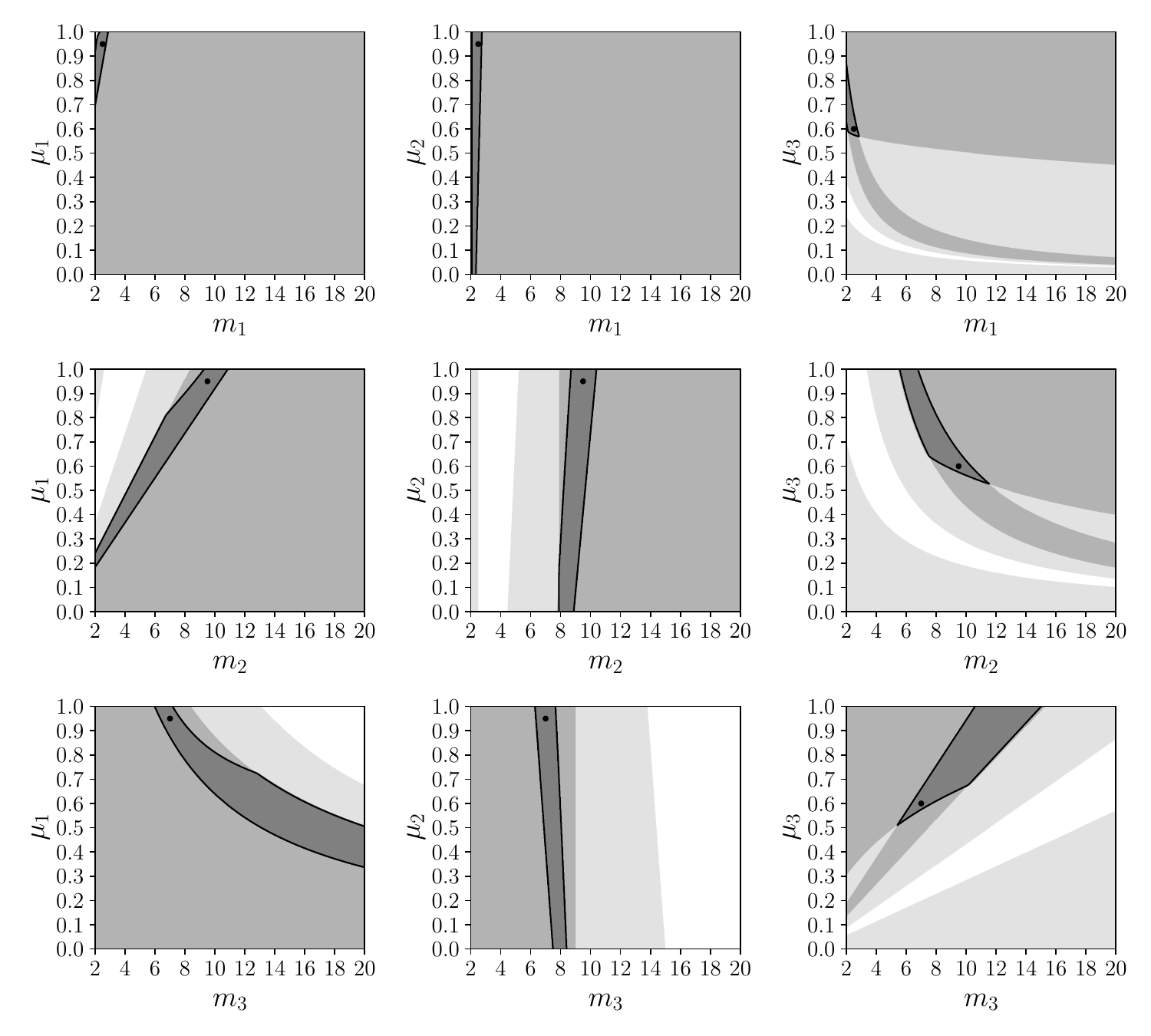}
    \caption{Cross sections of the six-dimensional parameter space, $\mathcal P$, in the $(m_i, \mu_j)$-planes, $i, j = 1, 2, 3$.
    In each panel, the regions $\mathcal R_1,\dots,\mathcal R_4$ are superimposed.
    The overlap, enclosed by a black curve, corresponds to parameters satisfying all constraints simultaneously. 
    The black dot marks the feasible point $p_\star$.}
    \label{fig:paramspace}
\end{figure}

The bifurcation structure of system~\eqref{eq:toysys} can be characterized using the Routh--Hurwitz polynomial $p_3$, whose roots determine stability boundaries in the diffusion parameter plane. 
The Turing-unstable region consists of all pairs of diffusion coefficients for which at least one spatial mode becomes unstable.
For each spatial eigenmode $\lambda_j$, define
$$\Gamma_j := \biggl\{ (D_v, D_w) \in (\RR_{>0})^2 \ : \ p_3(\lambda_j; D_v, D_w) < 0 \biggr\}, \qquad j\in\NN,$$ 
where the dependence of $p_3$ on $D_v$ and $D_w$ is made explicit.
Here, $\partial \Gamma_j$ denotes the associated stability curve $p_3(\lambda_j; D_v,D_w)=0$. Each $\Gamma_j$ consists of diffusion pairs destabilizing the $j$-th mode. If $(D_v,D_w) \in \Gamma_{i_1}\cap\dots\cap\Gamma_{i_k}$, the linearized operator at $\CSS_+$ has at least $k$ unstable eigenvalues. 
The full Turing-unstable set is therefore the union
\begin{align*}
    \Gamma := \bigcup_{j=1}^\infty \Gamma_j.
\end{align*}

Any diffusion pair $(D_v, D_w) \in \Gamma$ induces DDI of the steady state~$\CSS_+$. Figure~\ref{fig:bifurcation_diagram} illustrates the structure of $\Gamma$ and its constituent regions.
This description permits selective pattern control: choosing $(D_v,D_w)$ within a prescribed $\Gamma_j$ perturbs the $j$-th spatial mode, so that, with suitable initial conditions, patterns with a predictable number of peaks can be generated. Representative simulations on $\Omega=(0,1)$ are shown in Figure~\ref{fig:sim_patterns}.

\begin{figure}[t]
    \centering    \includegraphics[width=1.0\textwidth]{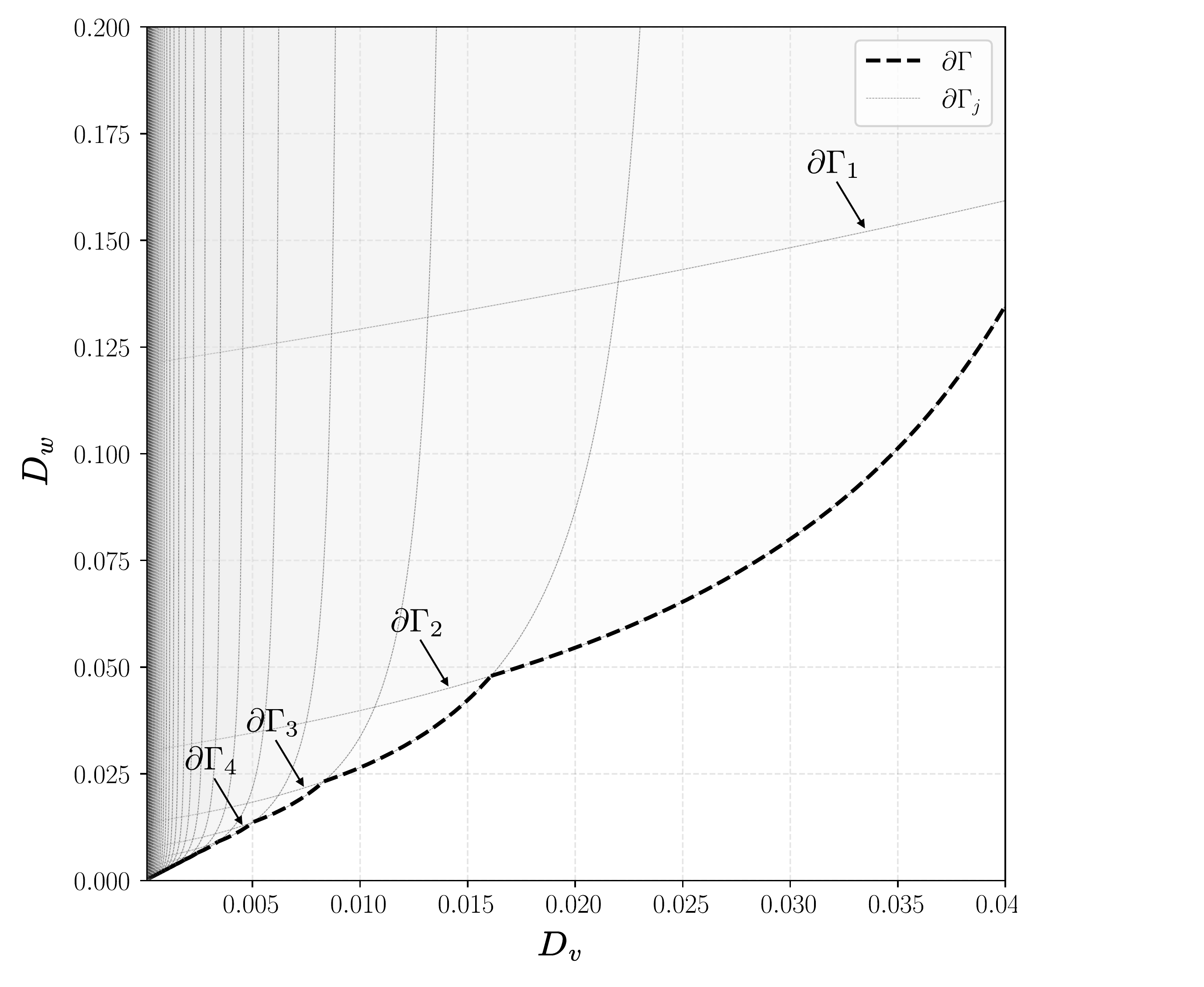}
    \caption{ Bifurcation diagram for system \eqref{eq:toysys} in the $(D_v, D_w)$-plane. The shaded region $\Gamma$ is the Turing unstable set, given by the union of all mode-specific regions $\Gamma_j$. \newline
    \textbf{Parameters}: $\mf=1.00$, $\ml=1.00$, $\me=0.60$, $m_1=2.50$, $m_2=9.68$, $m_3=7.00$.}
    \label{fig:bifurcation_diagram}
\end{figure}

\begin{figure}[ht]
\includegraphics[width=0.95\textwidth]{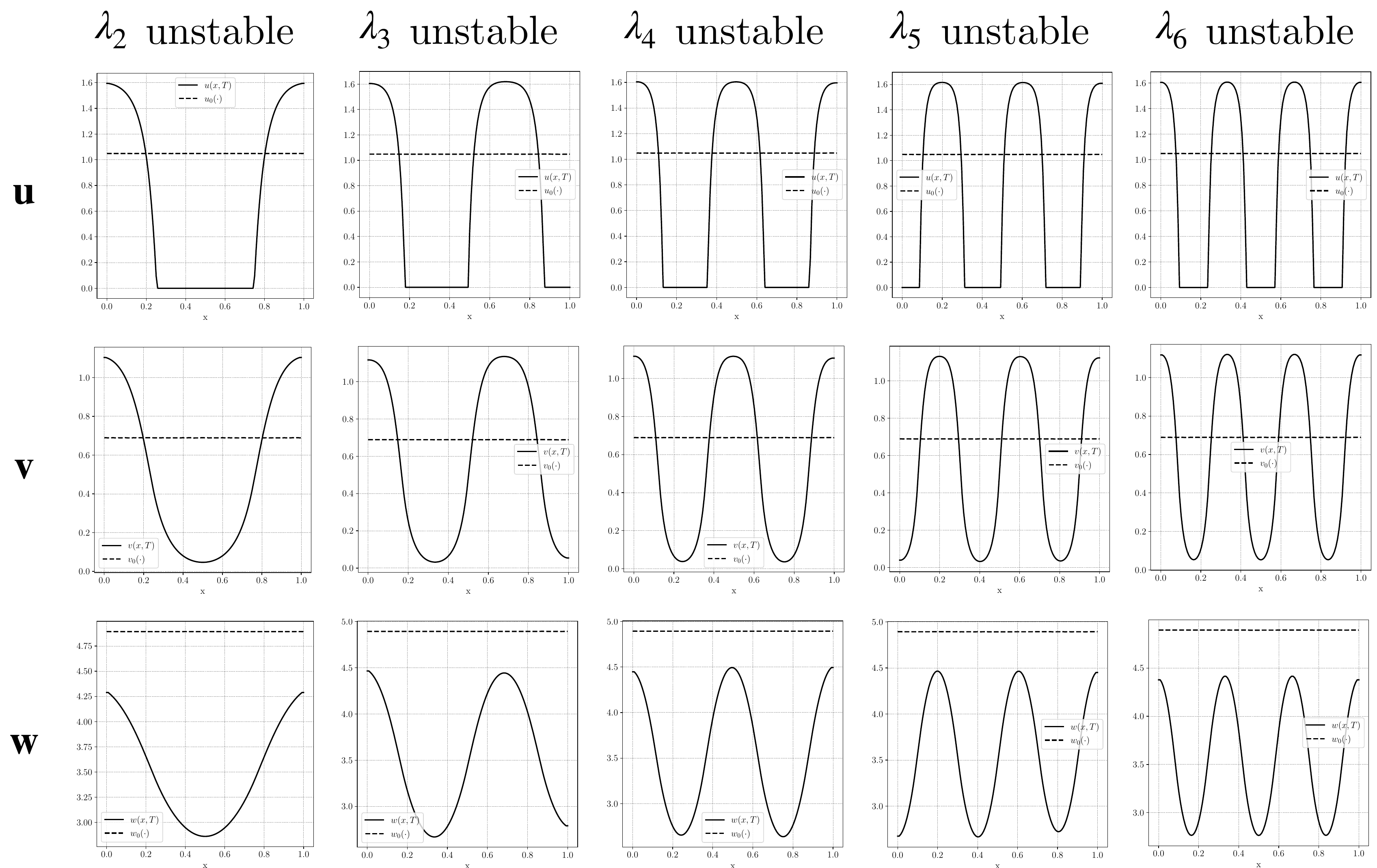}
    \caption{Patterned solutions to system \eqref{eq:toysys}.
    Each row corresponds to one component of $X=(u,v,w)$; each column shows the outcome for a different diffusion pair $(D_v,D_w)$. 
    In all cases, only a single spatial mode is unstable. \newline
    \textbf{Parameters}: $\mf=1.00$, $\ml=1.00$, $\me=0.60$, $m_1=2.50$, $m_2=9.68$, $m_3=7.00$. \newline
    \textbf{Init. Cond.}: $u_0 = \bar u + \xi$, \ \ $v_0 = \bar v + \xi$, \ \ $w_0 = \bar w + \xi$, \qquad $\xi = \frac{x}{10} \sin(10\pi x)$.} 
    \label{fig:sim_patterns}
\end{figure}

\subsection{Turing pattern vs.\ far-from-equilibrium patterns}

We distinguish between two fundamentally different types of spatial patterns. First, there are \textit{Turing patterns}, which
are
defined as regular structures that bifurcate from a constant steady state through DDI.
Their spatial scale is dictated by the unstable eigenmodes of the linearized operator.
In contrast, we observe \textit{far-from-equilibrium patterns}, {which} do not rely on DDI: they may exist as isolated solution branches, exhibit discontinuities or spikes, and form a continuum of stationary states depending on the nonlinear reaction terms.
Distinguishing between these pattern types can be challenging, as both types of patterns can evolve from a perturbation of a constant steady state exhibiting DDI.

The patterns observed in our simulations of system \eqref{eq:toysys} highlight the connection between classical Turing instability and far-from-equilibrium dynamics. To understand their nature, recall the nullcline for the nondiffusive component, $\{(u, v) \in \RR^2 : f (u, v) = 0 \}$, consists of two branches:
\begin{align*}
    u = 0 \quad \text{and}\quad u = \frac{m_1}{\mu_1} - \frac{1}{v},
\end{align*}
and the constant steady state $\bar X_+$ exhibiting DDI lies on the nontrivial branch. 
When perturbed by small random noise, the emerging patterns are initially governed by the unstable eigenmodes of the Laplacian: if only the $j$-th eigenmode is unstable, the spatial structure behaves similarly to $\cos(j\pi x/L)$ with the corresponding number of peaks, see Fig.~\ref{fig:sim_patterns}.

\begin{figure}[t]
\centering
\begin{subfigure}{0.9\linewidth}
    \centering
    \includegraphics[width=0.9\linewidth]{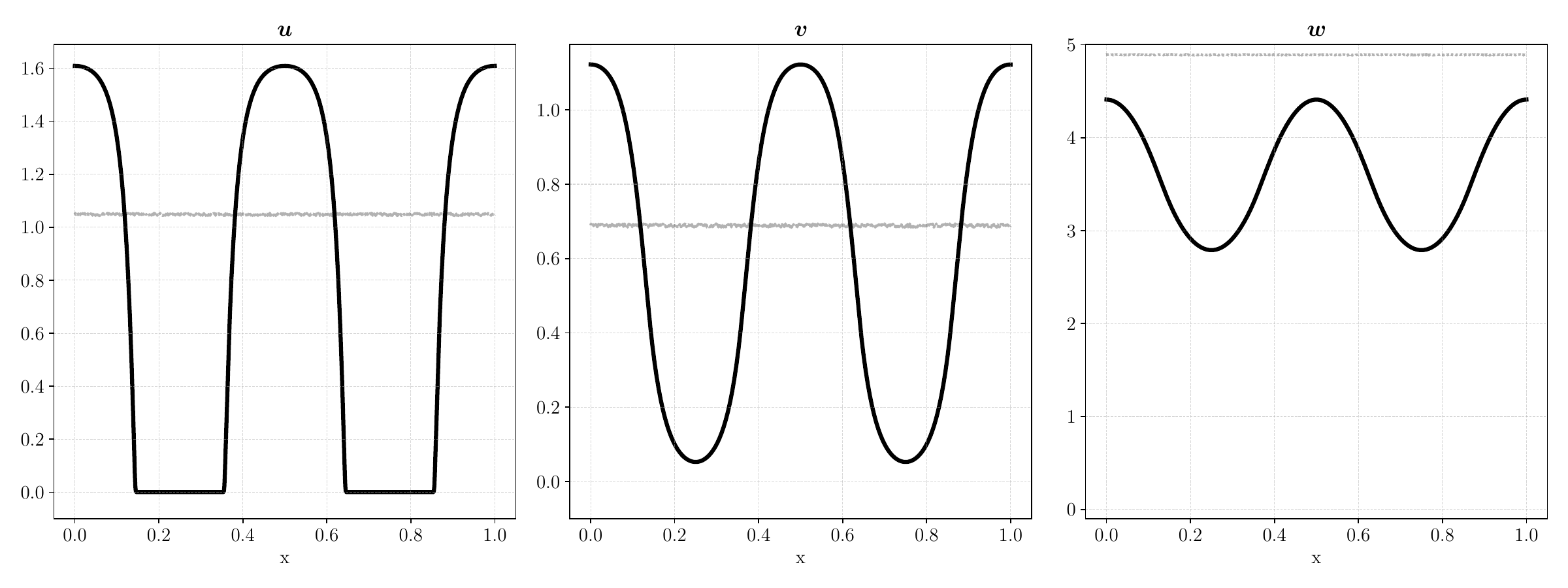}
    \caption{Small random perturbation around $\CSS$:
    \newline
    \textbf{Init. Cond.}: $u_0 = \bar u + \xi$, \ \ $v_0 = \bar v + \xi$, \ \ $w_0 = \bar w + \xi$, \qquad $\xi = 0.01 \cdot \mathcal U(-1, 1)$
    }
\end{subfigure} \\
\begin{subfigure}{0.9\linewidth}
    \centering
    \includegraphics[width=0.9\linewidth]{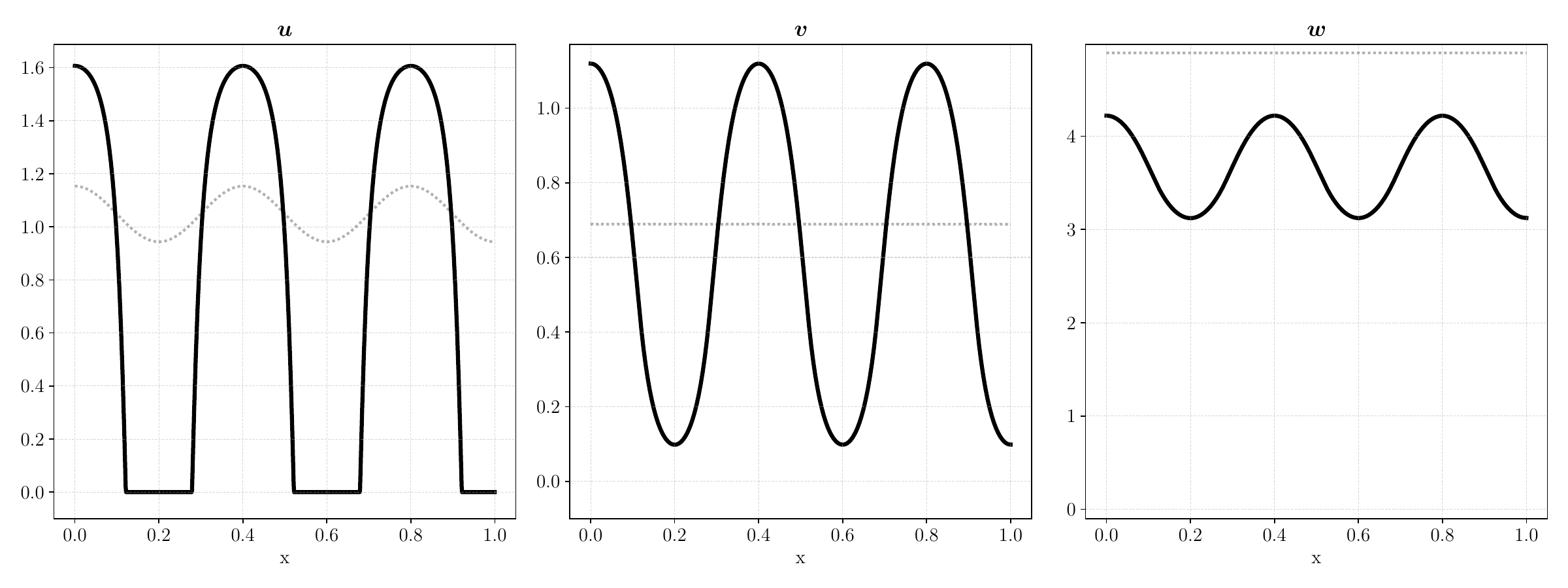}
    \caption{Perturbation describing eigenmode $\lambda_5$:
    \newline
    \textbf{Init. Cond.}: $u_0 = \bar u + 0.1 \cdot \bar u \cos(5 \pi x)$, \ \ $v_0 = \bar v + \xi$, \ \ $w_0 = \bar w + \xi$, \qquad $\xi = 0.01 \cdot \mathcal U(-1, 1)$
    }
\end{subfigure} \\
\begin{subfigure}{0.9\linewidth}
    \centering
    \includegraphics[width=0.9\linewidth]{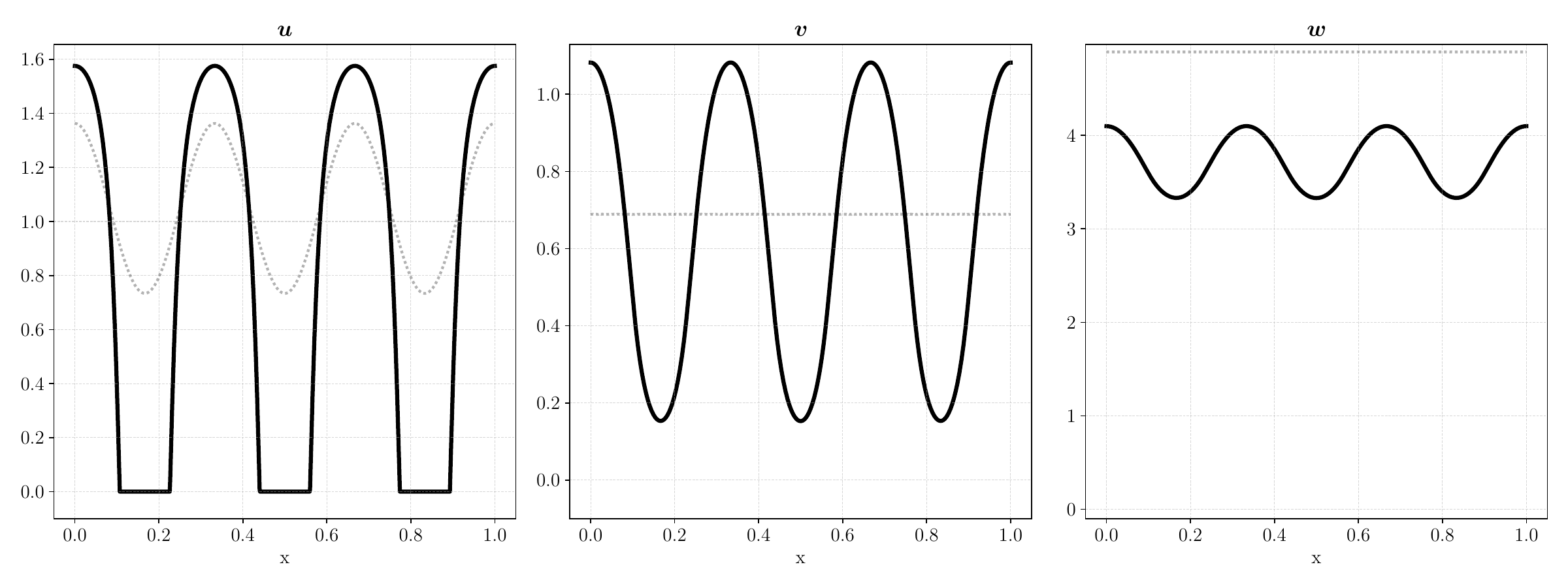}
    \caption{Perturbation describing eigenmode $\lambda_6$:
    \newline
    \textbf{Init. Cond.}: $u_0 = \bar u + 0.1 \cdot \bar u \cos(6 \pi x)$, \ \ $v_0 = \bar v + \xi$, \ \ $w_0 = \bar w + \xi$, \qquad $\xi = 0.01 \cdot \mathcal U(-1, 1)$
    }
\end{subfigure}
\caption{
Patterns from simulations with $D_v = 0.006$ and $D_w = 0.017$, 
for which only the eigenmode $\lambda_4$ is unstable. 
Different initial conditions select patterns associated with different eigenmodes. 
\newline
\textbf{Parameters}: $\mf=1.00$, $\ml=1.00$, $\me=0.60$, $m_1=2.50$, $m_2=9.68$, $m_3=7.00$.
}
\label{fig:different_perturbations}
\end{figure}

Surprisingly, nonlinear interactions allow stable patterns associated with higher modes that are \emph{linearly stable} at $\CSS_+$. Figure~\ref{fig:different_perturbations} shows that, when $\lambda_4$ is the only unstable mode, large-amplitude perturbations aligned with $\cos(5\pi x)$ or $\cos(6\pi x)$ in the nondiffusive component still yield patterns reflecting those modes. 
This demonstrates that nonlinearity can stabilize patterns which otherwise would be beyond the reach of classical Turing instability,
provided the perturbation is strong enough to escape the basin of attraction of the pattern selected by the unstable eigenmode.

A main observation that distinguishes patterns presented in Figure \ref{fig:sim_patterns} and Figure \ref{fig:different_perturbations} from classical Turing patterns is that rather than stabilizing near the constant steady state, solutions drift toward the trivial branch $u\equiv0$ of the nullcline, where they settle into stable far-from-equilibrium patterns. 
This branch-switching mechanism explains both the non-smooth structure of the $u$ state variable in Fig.~\ref{fig:sim_patterns} (despite their Turing-like appearance), and the discontinuous stationary states constructed in Fig.~\ref{fig:towerpattern1}. 
Thus, while unstable eigenmodes originating from DDI act as a selection mechanism, the observed patterns are ultimately far-from-equilibrium states determined by the structure of the nonlinear branch.
We recreate this property by varying initial conditions, see Figure \ref{fig:different_perturbations} for comparison.

Theorem~\ref{DisExBan} provides the theoretical foundation: for any sufficiently small subset $\Omega_2\subset\Omega$, one can construct stationary solutions in which $u$ switches between the two nullcline branches.
These patterns do not require DDI for their existence, only the presence of multiple solution branches satisfying Assumptions \ref{ass:StatSol}--\ref{ass:DiscontinuousStationaryTwoBranches}.
This yields uncountably many distinct far-from-equilibrium patterns, in stark contrast to the discrete set of Turing modes in classical reaction--diffusion systems. 
The continuum of possibilities arises because $\Omega_2$ may be chosen arbitrarily, producing a rich family of spatial organizations not captured by bifurcation from a steady state.
In fact, it is possible to construct patterns that are arbitrarily close in the $L^{2}$-norm. Consequently, none of these patterns can be \emph{asymptotically} stable in $L^{2}$, see Remark~\ref{rem:Mult_stab}.

Figure~\ref{fig:towerpattern1} illustrates this mechanism explicitly. Starting from a continuous pattern obtained near DDI, we force $u$ to jump to the trivial branch on small subsets of the domain. The resulting stationary profiles exhibit genuine jump discontinuities in $u$, showing that such patterns are not artifacts of numerics but arise from the system’s dynamics.

These stationary solutions are, however, unstable. Along the trivial branch $u\equiv0$, the nondiffusive submatrix of the Jacobian $J$ is $J_1 = \nabla_uf(0,v,w) = -\mu_1 + m_1 v$, which has a positive eigenvalue whenever $v>\mu_1/m_1 = 0.4$ for our choice of parameters. 
Since this threshold is exceeded on part of the subset where $u=0$ in Figure~\ref{fig:towerpattern1}, $J_1$ has a positive
eigenvalue on a set of positive measure there, so instability of the stationary solution follows from~\eqref{eq:spectrum_mult_operator} and~\eqref{eq:Spectrum_of_general_L}.
The pattern can nonetheless be produced numerically by, in addition to setting the initial condition to zero on the relevant subinterval, also suppressing the reaction term $f$ there, which vanishes identically along the branch $u\equiv0$. The solution is then not driven away from the unstable branch $u\equiv 0$.

In the case of stable Turing patterns, one would expect that choosing diffusion coefficients closer to the threshold for DDI would yield patterns of smaller amplitude, as the system is near to the bifurcation point. In such a regime, the patterns should arise purely on the nontrivial branch of the nullcline $f=0$ not using any far-from-equilibrium dynamics. However, this behavior is never observed in our simulations.

In conclusion, although the patterns in system~\eqref{eq:toysys} may resemble classical Turing patterns, their origin is fundamentally different: unstable eigenmodes associated with DDI act only as a catalyst, guiding the solution toward far-from-equilibrium states created by branch switching in the nondiffusive component.
This hybrid mechanism enriches the pattern formation landscape and highlights phenomena that cannot occur in standard reaction--diffusion systems.

\begin{figure}[t]
    \centering
    \begin{subfigure}{0.9\linewidth}
        \centering
        \includegraphics[width=0.9\linewidth]{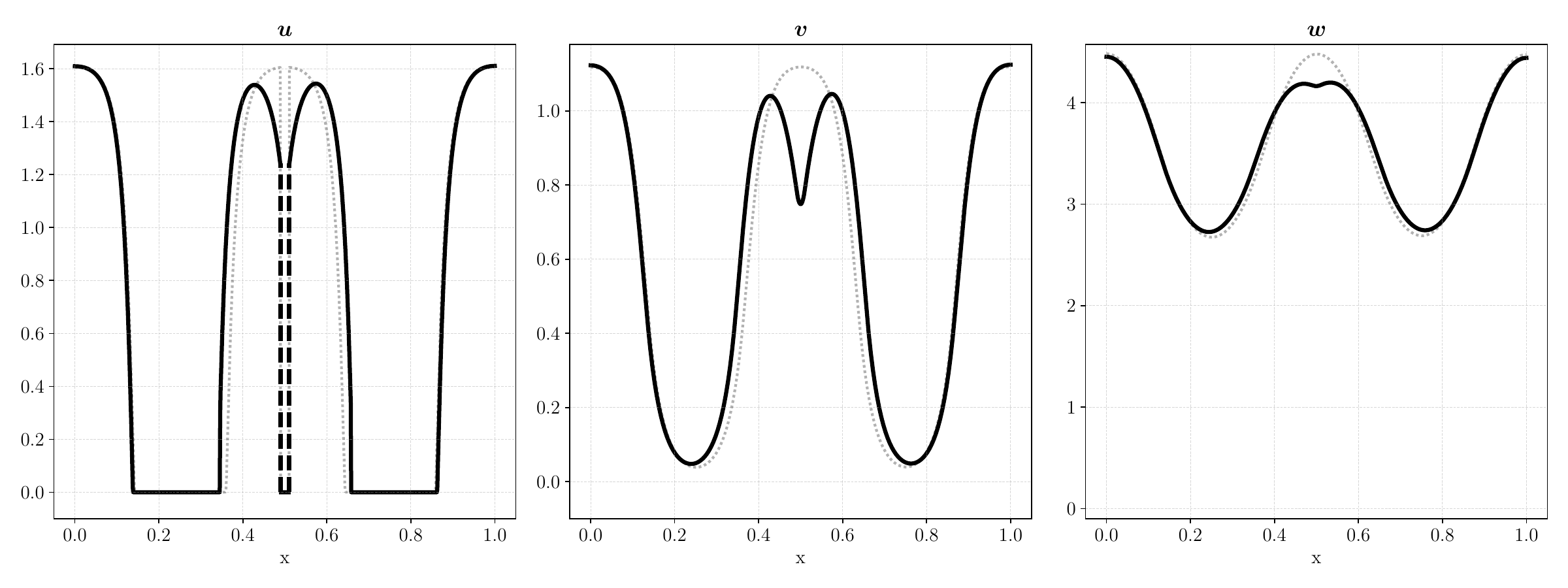}
        \caption{Initial profile forced to zero on $(0.49,0.51)$ produces two jump discontinuities at $x = 0.49$ and $x = 0.51$ (dashed black).
        }
    \end{subfigure}
    \begin{subfigure}{0.9\linewidth}
        \centering
        \includegraphics[width=0.9\linewidth]{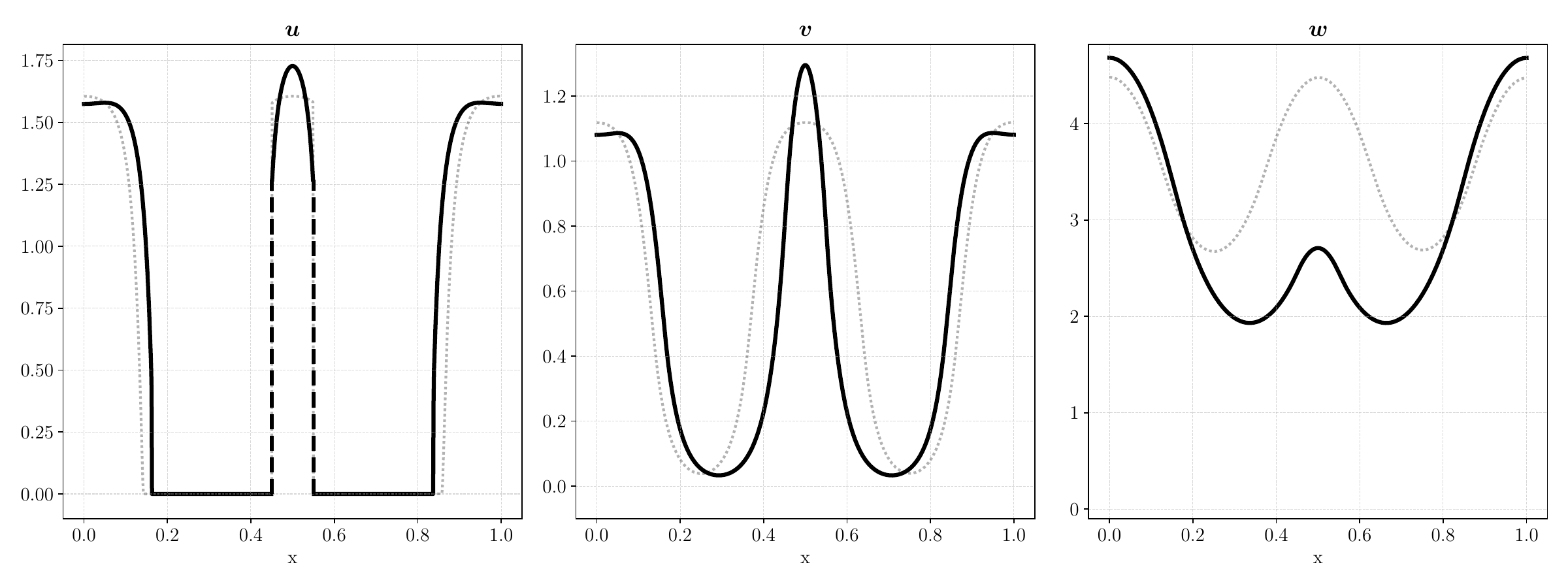}
        \caption{Forcing $u=0$ on two subintervals $(0.35, 0.45)$ and $(0.55, 0.65)$ produces jumps at $x = 0.45$ and $x = 0.55$ (dashed black).
        }
    \end{subfigure}
    \caption{
    Simulation of System~\eqref{eq:toysys} with initial condition (dotted gray) taken from the pattern in Figure~\ref{fig:sim_patterns} ($\lambda_4$ unstable) and modified by introducing an artificial jump to zero as indicated in the subfigures. The resulting stationary profile (black) exhibits jump discontinuities.
    \newline
    \textbf{Parameters}: $\mf=1.00$, $\ml=1.00$, $\me=0.60$, $m_1=2.50$, $m_2=9.68$, $m_3=7.00$.
    }
    \label{fig:towerpattern1}
\end{figure}

\section*{Acknowledgment}
{This work is supported by the German Research Foundation (DFG) under Germany’s Excellence Strategy EXC 2181/1 - 390900948 (the Heidelberg STRUCTURES Excellence Cluster) and through the Collaborative Research Center 1324 (SFB1324, project B6). Additional funding was provided by the European Research Council (ERC) under the Synergy Grant (PEPS, No. 10107178)}

\bibliographystyle{abbrv}
\bibliography{biblio.bib}

\appendix

\section{Proof of Theorem \ref{thm:bistabletoysys}}
\label{app1}

We now prove Theorem \ref{thm:bistabletoysys}. The argument is divided into three steps. 
First, we show that the system admits three steady states when condition \ref{item:sec_assumption} holds.  
Second, we establish that the two nontrivial steady states are positive under condition \ref{item:first_assumption}.  
Finally, we analyze the Jacobian to determine that $\CSS_-$ is unstable and that $\CSS_+$ is locally asymptotically stable under condition~\eqref{eq:stab_cond_X+}.

\begin{proof} \textbf{Step I: Analysis of steady states.} 
It is clear that $\CSS_0 := (0,0,0)$ is a constant steady state and the only one with $\bar u = 0$. At equilibrium, the relation $\frac{\del \bar u}{\del t} = 0$ implies $\bar u = \eta_1 - \frac{1}{\bar v}$ for $\bar u \ne 0$, and substituting the expression of $\bar u$ into $\frac{\del \bar w}{\del t} = 0$ gives
\mbox{$\bar w = \eta_3 (1 - \frac{1}{\eta_1 \bar v}).$}
Plugging both relations into $\frac{\del \bar v}{\del t} = 0$ yields a quadratic equation that $\bar v$ must obey:
\begin{align}\label{eq:steadystatev}
    \eta_1(\mu_2 + \eta_3) \bar v^2 - (\alpha + 2\eta_3) \bar v + m_2 = 0.
\end{align}
The discriminant
\begin{equation}\label{eq:existssteadystate}
(\alpha + 2\eta_3)^2 - 4 m_2\eta_1 (\mu_2 + \eta_3) = (\alpha-2\mu_2)^2 - 4 \mu_2 (\mu_2 + \eta_3)
\end{equation}
is positive thanks to  assumption \ref{item:sec_assumption}.
Thus, equation \eqref{eq:steadystatev} has two distinct real solutions and overall, system \eqref{eq:toysys} admits three constant steady states: the trivial equilibrium $\CSS_0$,  and two nontrivial steady states $\CSS_\pm = (\eta_1 - \frac{1}{\bar v_\pm}, \bar v_\pm, \eta_3  - \frac{\eta_3}{\eta_1 \bar v_\pm}),$ where
\begin{align*}
\bar v_\pm := \dfrac{\alpha + 2\eta_3 \pm  \sqrt{(\alpha-2\mu_2)^2 - 4 \mu_2 (\mu_2 + \eta_3)}}{2 \eta_1 (\mu_2 + \eta_3)}.
\end{align*}

\vspace{1em}
\noindent
\textbf{Step II: Positivity of the nontrivial steady states.}
It is straightforward to verify $\bar v_\pm > 0$ for all admissible parameters.
From $\bar u_\pm = \eta_1 - 1/\bar v_\pm$ and $\bar w_\pm = \eta_3 (1 - 1/(\eta_1 \bar v_\pm))$, both $\bar u_\pm,$ and $\bar w_\pm$ are positive, if $\eta_1 \bar v_\pm > 1$.
Substituting the expression for $\bar v_\pm$ gives the inequality
\be\label{eq:condition_positive_+s}\mp \sqrt{(\alpha-2\mu_2)^2 - 4\mu_2(\mu_2 + \eta_3)} < \alpha - 2\mu_2.\ee
Recall that condition \ref{item:first_assumption} ensures $\alpha - 2\mu_2 > 0$. 
For $\bar v_+$, inequality \eqref{eq:condition_positive_+s} holds automatically since the left-hand side is negative and the right-hand side is positive.
For $\bar v_-$, both sides of equation \eqref{eq:condition_positive_+s} are positive, and the inequality \eqref{eq:condition_positive_+s} holds since $\Theta < \zeta^2$, so $\sqrt{\Theta}<\zeta$. Therefore, both $\CSS_\pm$ are positive.

\vspace{1em}
\noindent
\textbf{Step III: Stability analysis.}
We investigate conditions under which the system is bistable. The Jacobian evaluated at an arbitrary equilibrium $\CSS = (\bar u,\bar v,\bar w)$ is
$$ \hspace{-2em} J(\CSS) = \dfrac{1}{(1+ \bar u \bar v)^2}
\bpm
-\mf (1+ \bar u \bar v)^2 + m_1 \bar v & m_1 \bar u & 0 \\
m_2 \bar v & -(\ml + \bar w) (1+ \bar u \bar v)^2 + m_2 \bar u & -\bar v (1+ \bar u \bar v)^2 \\
m_3 \bar v & m_3 \bar u & -\mu_3 (1+ \bar u \bar v)^2
\epm.$$
First, we observe $J(\CSS_0) = -\diag(\mf, \ml, \me),$ which is clearly stable.
However, more work is needed to conclude about the stability of $J$ around $\CSS_\pm$. Exploiting relations on the steady-states, we obtain
$$ J(\CSS_\pm) = \dfrac{1}{(1 + \bar u_\pm \bar v_\pm)^2} \bpm -m_1 \bar u_\pm \bar v_\pm^2 & m_1 \bar u_\pm  & 0 \\ m_2 \bar v_\pm & -m_2 \bar u_\pm ^2 \bar v_\pm & -\bar v_\pm (1 + \bar u_\pm \bar v_\pm)^2 \\ m_3 \bar v_\pm & m_3 \bar u_\pm  & -\me (1 + \bar u_\pm \bar v_\pm)^2 \epm.$$
Our next step is to use the Routh--Hurwitz criterion to show that $J$ is unstable around $\bar X_-$, and stable around $\bar X_+$. 

Computing $\det(J(\CSS_\pm))$ yields
\begin{align*}
\det(J(\CSS_\pm)) = \dfrac{1}{(1 + \bar u_\pm \bar v_\pm)^6} \biggl[ \bar v_\pm & (1 + \bar u_\pm \bar v_\pm)^2 \biggl(-m_1m_3 \bar u_\pm^2 \bar v_\pm^2 - m_1 m_3 \bar u_\pm \bar v_\pm \biggr) \\ & - \mu_3 (1 + \bar u_\pm \bar v_\pm)^2 \biggl( m_1 m_2 \bar u_\pm^3 \bar v_\pm^3 - m_1 m_2 \bar u_\pm \bar v_\pm \biggr) \biggr].
\end{align*}
Using the definition of $\bar u_\pm$, we establish $\bar u_\pm \bar v_\pm - 1 = \eta_1 \bar v_\pm - 2$. 
Thus,
\begin{align}
\label{eq:detJ}
\det(J(\CSS_\pm)) = -m_1 \mu_3 \dfrac{\bar u_\pm \bar v_\pm}{(1 + \bar u_\pm \bar v_\pm)^3} \biggl[\bigl(\alpha + 2 \eta_3 \bigr) \bar v_\pm  - 2 m_2\biggr].
\end{align}
For convenience, we introduce the notation
\begin{align*}
&\beta := \alpha + 2\eta_3 > 0, & &\gamma := 4\mu_2 (\alpha + \eta_3) > 0, & &\delta := 4 \eta_3 (\alpha + \eta_3) > 0.
\end{align*} 
Substituting the expression of $\bar v_\pm$ 
into \eqref{eq:detJ} yields
$$-\det(J(\CSS_\pm)) = m_1 \me \dfrac{\bar u_\pm \bar v_\pm}{(1+\bar u_\pm \bar v_\pm)^3} \lp \dfrac{\beta}{2\eta_1 (\ml + \eta_3)} \lp \beta \pm \sqrt{\alpha^2 - \gamma} \rp - 2 m_2 \rp.$$
We stress that the square roots have positive argument, since $\alpha^2 - \gamma > 0$ by assumption \ref{item:sec_assumption}.
Factoring the $2\eta_1(\ml + \eta_3)$ term out, we obtain
$$- \det(J(\CSS_\pm)) = \dfrac{\mf \me}{2 (\ml + \eta_3) } \dfrac{\bar u_\pm \bar v_\pm}{(1+\bar u_\pm \bar v_\pm)^3} \lp \beta^2 - \delta - \gamma \pm \beta \sqrt{\alpha^2 - \gamma} \rp.$$
Using $\beta^2 - \delta = \alpha^2$, we see that $-\det(J(\CSS_\pm))$ has the same sign as
$$\mathcal A := \alpha^2 - \gamma \pm \beta \sqrt{\alpha^2 - \gamma} = \sqrt{\alpha^2 - \gamma} \lp \sqrt{\alpha^2 - \gamma} \pm \beta \rp.$$
We deal with $\bar v_+$ and $\bar v_-$ separately. For $\bar v_-$, it holds
$$\mathcal A = \dfrac{\sqrt{\alpha^2 - \gamma}}{\sqrt{\alpha^2-\gamma} + \beta} \lp \alpha^2 - \beta^2 - \gamma \rp = - \dfrac{\sqrt{\alpha^2 - \gamma}}{\sqrt{\alpha^2-\gamma} + \beta} \lp \delta + \gamma \rp < 0,$$
so $-\det(J(\CSS_-)) < 0$ and the Routh--Hurwitz criterion enables us to conclude that $\CSS_-$ is unstable. On the other hand, for $\bar v_+$, $\mathcal A > 0$ since $\alpha^2-\gamma>0$ by~\eqref{eq:existssteadystate} and $\beta>0$. Thus $-\det(J(\CSS_+)) > 0$.
Using the Routh--Hurwitz criterion once again, the steady state $\CSS_+$ is stable if
\be\label{cond:DDI}-\tr \bigl( J(\CSS_+) \bigr) \biggl(\sum_{1 \leqslant i < j \leqslant 3} \det \bigl( J(\CSS_+)_{ij} \bigr) \biggr) + \det \bigl( J(\CSS_+) \bigr) > 0.\ee
This last expression is rather involved, so we break it down into smaller parts. Using the positivity of all terms in $-\tr(J(\CSS_+))$, we obtain the (coarse) estimate $-\tr(J(\CSS_+)) > \mu_3$.
Next, by distributing $\frac{m_1}{(1 + \bar u_+ \bar v_+)}$ inside the parenthesis in equation \eqref{eq:detJ}, the coefficient $-\det(J(\CSS_+))$ is rewritten as
$$-\det(J(\CSS_+)) = \mu_3 \dfrac{\bar u_+ \bar v_+}{(1 + \bar u_+ \bar v_+)^2} \lp \dfrac{m_1 \eta_3 \bar v_+}{1 + \bar u_+ \bar v_+} - m_1 m_2 \dfrac{1- \bar u_+ \bar v_+}{1 + \bar u_+ \bar v_+}\rp.$$
Finally, a direct computation shows
$$\sum_{1 \leqslant i < j \leqslant 3} \det \bigl( J(\CSS_+)_{ij} \bigr) = \mu_3 \dfrac{\bar u_+ \bar v_+}{(1 + \bar u_+ \bar v_+)^2} \lp m_1 \bar v_+ + m_2 \bar u_+ + \eta_3 - \dfrac{m_1 m_2}{\mu_3} \dfrac{1- \bar u_+ \bar v_+}{1 + \bar u_+ \bar v_+} \rp.$$
Putting everything together, inequality \eqref{cond:DDI} holds if
$$m_1 \bar v_+ + m_2 \bar u_+ + \eta_3 > \dfrac{m_1 \eta_3}{(1 + \bar u_+ \bar v_+) \mu_3} \bar v_+.$$
Since $1 + \bar u_+ \bar v_+ = \eta_1 \bar v_+$, the right-hand side equals $\mu_1\eta_3/\mu_3$, so this inequality is exactly condition~\eqref{eq:stab_cond_X+}.
Hence, $\CSS_+$ is locally asymptotically stable under~\eqref{eq:stab_cond_X+}, and the proof is complete.
\end{proof}

\section{Existence of a nonnegative, bounded, global solution to System (\ref{eq:toysys})}
\label{app2}

In this section, we show that the system \eqref{eq:toysys} has a global solution, which is nonnegative and bounded, using a bounded invariant region argument. Kowall et al.~\cite[Proposition C.2]{KMMü23} recently established that RD--ODE systems of the form \eqref{eq:toysys} admit a unique local-in-time mild solution.
We denote such a solution by $X = (u,v,w)$.

Our goal here is to extend this local solution to all times, by proving that solutions to \eqref{eq:toysys} remain \textit{a priori} bounded for all time.
This is achieved using the theory of bounded invariant regions (see ~\cite[Chap. 14, \S B]{Smoller1983}). We first recall the definition of an invariant region.

\begin{definition}[Invariant Region, \cite{Smoller1983}] A closed subset $\Sigma \subset \RR^3$ is called a (positively) invariant region for the initial value problem \eqref{eq:toysys} if any solution $X(x,t)$ with initial values in $\Sigma$ satisfies $X(x,t) \in \Sigma$ for all $x \in \Omega$ and for all $t \in [0, T_{\max})$.
\end{definition}

It is well known that rectangular regions are particularly convenient for reaction--diffusion systems.
By \cite[Corollary 14.8, Theorem 14.11]{Smoller1983}, the hyper-rectangle 
$$\Sigma := \biggl\{ X = (u, v, w) \in \RR^3 \ : \ a_1 \leqslant u \leqslant b_1, \quad a_2 \leqslant v \leqslant b_2, \quad a_3 \leqslant w \leqslant b_3 \biggr\}$$
is invariant provided the nonlinear mapping $F$ points strictly into $\Sigma$ on $\del\Sigma$. 
If in addition the system \eqref{eq:toysys} has the $f-$stable property (see \cite[Definition 14.10]{Smoller1983}), it is sufficient that $F$ points non-strictly into $\Sigma$ on the boundary. 

The $f-$stability of system $\eqref{eq:toysys}$ follows by adapting the proof in \cite[Theorem 3]{Daub2013Nov}.
Thus, it is only necessary to verify that $F$ points inward on the boundary of $\Sigma$.

\begin{proof}
    Let $A_1, A_2, A_3 > 0$ denote arbitrary constants, and define the rectangular region
    \begin{align*}
        \Sigma := \bigcap_{i=1}^3[0, A_i] \subset \RR^3.
    \end{align*}
    We introduce the six edge functions
    \begin{align*}
        G_1(u, v, w) &= -u, & G_2(u, v, w)&= -v, & G_3(u, v, w) &= -w, \\[0.7em]
        G_4(u, v, w) &= u - A_1, & G_5(u, v, w)&= v - A_2, & G_6(u, v, w) &= w - A_3,
    \end{align*}
    so that $\Sigma = \bigcap_{j=1}^6\{ G_j \leqslant 0\}$.
    By definition, for any boundary point $(u_\star, v_\star, w_\star) \in \del\Sigma$, there exists a $j$ such that $G_{j}(u_\star,v_\star,w_\star) = 0$.
    The condition that the nonlinear mapping $F$ points into $\Sigma$ is 
    $$\grad G_{j}(u_\star, v_\star, w_\star) \cdot F(u_\star, v_\star, w_\star) \leqslant 0.$$
    
    If $u_\star = 0$, $v_\star = 0$, or $w_\star = 0$, then $f(u_\star, v_\star, w_\star) = 0$, $g(u_\star, v_\star, w_\star) = 0$, or $-h(u_\star, v_\star, w_\star) \leqslant 0$ respectively and the condition is satisfied for $j=1,2,3$. This already shows that solutions remain nonnegative. On the other hand, if $u_\star = A_1$, or $v_\star = A_2$ or $w_\star = A_3$, then $\grad G_{j} = e_{j-3}$ for $j=4,5,6$.
    Using the positivity of $u,v,w$ and the bound $\frac{uv}{1+uv} \leqslant 1$, the following estimates follow
    \begin{align*}
    \hspace{-1em} f(A_1, v, w) \leqslant m_1 - \mu_1 A_1, \quad 
    g(u, A_2, w) \leqslant m_2 - \mu_2 A_2, \quad 
    h(u, v, A_3) \leqslant m_3 - \mu_3 A_3.
    \end{align*}
    Therefore, by choosing $A_i \geqslant \max\{\eta_i, \|X_{0,i}\|_{L^\infty(\Omega)}\}$ for $i = 1, 2, 3$, where $X_{0,i}$ denotes the $i$-th component of the initial datum, it holds
    \begin{gather*}
    \hspace{-2em}\grad G_{4}(A_1, v, w) \cdot F(A_1, v, w) \leqslant 0, \quad \grad G_{5}(u, A_2, w) \cdot F(u, A_2, w) \leqslant 0,\\ \grad G_{6}(u, v, A_3) \cdot F(u, v, A_3) \leqslant 0,
    \end{gather*}
    and the nonnegative initial datum lies in $\Sigma$.
    Hence $\Sigma$ is a bounded invariant region for system \eqref{eq:toysys} containing the initial datum, so the solution remains in $\Sigma$ for all times of existence. In particular, it is bounded and therefore global by the blow-up alternative for the local mild solution~\cite{KMMü23}. The proof is complete.
\end{proof}



\end{document}